\documentclass[12pt,reqno]{amsart}
\usepackage{amsmath}
\usepackage{amsxtra}
\usepackage{amscd}
\usepackage{amsthm}
\usepackage{amsfonts}
\usepackage{amssymb}
\usepackage{eucal}
\usepackage{epsfig}
\usepackage{graphics}
\usepackage{accents}
\usepackage{simplewick}

\usepackage{mathtools}
\usepackage{tikz}
\usepackage[all]{xy}
\usepackage{here} 
\usetikzlibrary{calc}
\numberwithin{equation}{section}
\newtheorem{thm}{Theorem}[section]
\newtheorem{prop}[thm]{Proposition}
\newtheorem{lem}[thm]{Lemma}
\newtheorem{rem}[thm]{Remark}
\newtheorem{cor}[thm]{Corollary}

\newcommand{\nn}{\nonumber}

\theoremstyle{definition}

\newcommand{\ket}[1]{{| #1 \rangle}}      

\newcommand{\C}{{\mathbb C}}
\newcommand{\Z}{{\mathbb Z}}

\newcommand{\E}{{\mathcal E}}
\newcommand{\F}{\mathcal F}

\newcommand{\gl}{\mathfrak{gl}}

\newcommand{\slt}{\mathfrak{sl}_2}

\newcommand{\End}{\mathop{\rm End}}

\newcommand{\res}{{\rm res}}
\newcommand{\Res}{\mathop{\rm res}}

\newcommand{\Sym}{\mathrm{Sym}}

\newcommand{\wt}{{\rm wt}\,}

\newcommand{\on}{\operatorname}
\newcommand{\mc}{\mathcal}
\newcommand{\al}{\alpha}

\newcommand{\cq}{\check{q}}

\newcommand{\cb}{\check{\beta}}

\newcommand{\calA}{\mathcal{A}}

\newcommand{\bq}{\bar q}

\begin{document}

\begin{title}{Affinization of shifted  quantum affine $\gl_2$}
\end{title}

\author{B. Feigin, M. Jimbo, and E. Mukhin}

\address{BF: 
Hebrew University of Jerusalem, Einstein Institute of Mathematics,
Givat Ram. Jerusalem, 9190401, Israel
}
\email{bfeigin@gmail.com}

\address{MJ: 
Professor Emeritus,
Rikkyo University, Toshima-ku, Tokyo 171-8501, Japan}
\email{jimbomm1@gmail.com}

\address{EM: Department of Mathematics,
Indiana University Indianapolis,
402 N. Blackford St., LD 270, 
Indianapolis, IN 46202, USA}
\email{emukhin@iu.edu}

\dedicatory{To the memory of Kenji Iohara}


\begin{abstract} We give a realization $\mc A_0$ of quantum toroidal algebra associated to $\gl_2$ which can be viewed as an affinization of the Drinfeld new realization of quantum affine $\gl_2$. We use this realization to define an affinization $\mc A_N$, $N\in\Z$, of shifted quantum affine $\gl_2$. We construct a large family of representations of dominantly shifted algebra $\mc A_N$, $N>0$. The examples of representations with even positive $N$ appear in the study of extensions of deformed $W$-algebras of type $\gl(N+2|1)$.
\end{abstract}

\keywords{shifted quantum toroidal algebra, fused currents, Fock module \newline ${}$ \hspace{6pt} {\it AMS classification numbers:} 17B37 primary, 81R50, 17B67}

 \bigskip

\maketitle

\section{Introduction}
The standard description of a quantum toroidal algebra is given in terms of currents $E_i(z)$, $F_i(z)$, $\Psi_i^\pm(z)$, $i\in \hat I$, corresponding to Chevalley generators $E_i, F_i, H_i$ of an affine Lie algebra, where $\hat I$ denotes the set of vertices of an affine Dynkin graph. In particular, the zeroth components  $E_{i,0}$, $F_{i,0}$, $\Psi_{i,0}^\pm$, $i\in \hat I$, generate the quantum affine  algebra in the Drinfeld-Jimbo realization. It is very well-known that a quantum affine algebra has a Drinfeld new realization, where generators are  $X_{i,s}^\pm, k_{i,r}^\pm$, $i\in I$, $s\in\Z$, $r\in\Z_{\geq 0}$. A natural question is if the toroidal algebras can similarly be described in terms of some natural currents $X_{i,s}^\pm(z)$, $k_{i,r}^\pm(z)$, $i\in I$, $s\in \Z$, $r\in\Z_{\geq 0}$. In this paper we propose such a realization in the case of $\gl_2$.
In the case of  $\gl_2$, we have $I=\{1\}$ and we drop index $i$ to write $k_r^\pm(z)=k_{1,r}^\pm(z)$, $X^\pm_{s}(z)=X^\pm_{i,s}$, etc.

Our construction is given in terms of fused currents. By fused currents we mean  counterparts of currents appearing in operator product expansions in Conformal Field Theory. In the deformed case, such fused currents were used in \cite{FJMM3} to construct various subalgebras of (completions of) quantum toroidal $\gl_n$ algebras. In particular, the quantum toroidal $\gl_2$ algebra $\E_2$ has two mutually commuting subalgebras isomorphic to quantum toroidal $\gl_1$ (with different parameters), see Propositions \ref{gl1 sub prop}, \ref{commute prop}. We call these subalgebras $\E_1$ and $\check{\E}_1$. The subalgebra $\E_1\otimes \check{\E}_1\subset \E_2$ together with Cartan elements $\Psi_{0,0}, \Psi_{1,0}$ play 
the role of the ``Cartan subalgebra" in our construction. We use fused currents in $\E_1$ and $\check{\E}_1$ to produce currents $k^\pm_r(z)$ and  $\check{k}^\pm_r(z)$, $r\in\Z_{\geq 0}$.  We have two sets of Cartan currents $k^\pm_r(z)$ and  $\check{k}^\pm_r(z)$ since we work with $\gl_2$ as opposed 
to $\slt$.  

Then we define currents $X^\pm_s(z)$. Recall that in $U_q\widehat{\gl}_2$, elements $X^+_s$ are obtained from $E_1=X^+_0$ by taking commutators with $k_1^\pm$. We define operators $X_s^+(z)$ using the left adjoint action of either algebra  $\E_1$ or $\check{\E}_1$ on $E_1(z)=X_0^+(z)$ (up to a shift of $z$). Similarly,  we define operators $X_s^-(z)$ by the right adjoint action of either algebra $\E_1$ or $\check{\E}_1$ on $F_1(z)$.
It turns out that the span $\langle X_s^+(q_1^{-s}z)\rangle_{s\in \Z}\subset \E_2$ is identified with a vector representation $V_1(z)$ as in \cite{FFJMM} and \eqref{vector rep} for $\E_1$ and the span $\langle X_s^+(z)\rangle_{s\in \Z}$ with a vector representation  $\check{V}_1(z)$ of $\check{\E}_1$. A similar statement holds for $X_s^-(z)$, $s\in \Z$. Thus we can say that $\E_2$ is generated by $\E_1\otimes \check{\E}_1$ 
and a family of left vector representations and a family of right vector representations depending on a parameter $z$. We will explore this point further in \cite{FJM2}.

After the definition of the generating currents, we compute their relations and that gives us an
algebra $\mc A_0$ isomorphic to $\E_2[K]$ (we add an element $K$ to the Cartan algebra), see Theorems \ref{iota thm}, \ref{tau thm}.

We note that the relations in  $\mc A_0$ contain infinite series and therefore $\mc A_0$ is an algebra only in the sense of action on admissible representations, see Remark \ref{waiver remark}. We also note that zeroth components $X^\pm_{s,0}$ do not seem to give quantum affine $\slt$ in the Drinfeld new realization.

\medskip

An application and a motivation for our construction is a definition of the shifted version $\mc A_N$, where $N\in \Z$ is viewed as an integral $\slt$ weight. For that definition we retain the relations between $X^\pm_s(z)$ with $X^\pm_{s'}(w)$ and $\E_1,\check{\E}_1$ algebras (in other words 
we keep the two ``Borel" algebras unaltered) and change only the relations between $X^+_s(z)$ and $X^-_{s'}(z)$, in the same way as it is done in shifted quantum affine \cite{FT} and toroidal algebras \cite{N}, see \eqref{EFshifted}. In the toroidal case we also change the identification of central elements of $\E_1$ and $\check{\E}_1$.

The algebra $\mc A_N$ is different from the shifted quantum toroidal algebra of \cite{N}. However, we expect some similarities. In particular, if $M_1$ is an $\mc A_{N_1}$ module and if $M_2$ is an $\mc A_{N_2}$ module, then we expect that $M_1\otimes M_2$ should naturally be an $\mc A_{N_1+N_2}$ module. In the present paper we show that this is the case when $M_1$ is a special Fock module of $\mc A_1$. That allows us to construct a large family of representations of $\mc A_N$ with $N>0$.

\medskip

In \cite{FJM1} we studied the extensions of deformed $W$-algebras using combinatorics of unbalanced $qq$-characters. 
The extension was performed by adjoining currents, 
which carry non-trivial momenta and commute with the screening operators.
The present paper is in part motivated by the desire to understand the nature of algebras of \cite{FJM1}.
It turns out that the extension of deformed $W$-algebra of type $\gl(N+2|1)$ with even $N$ coincides with algebra $\mc A_N$ (acting in a representation). Here the image of  $\E_1$ coincides with  the deformed $W$ algebra. The algebra $\check{\E}_1$ provides an extra boson used in  \cite{FJM1} to rationalize the currents, while the currents $X^\pm_0(z)$ are the additional currents in the extension of the deformed $W$-algebra. 

\medskip

We hope to generalize the results of this paper for quantum toroidal algebra of type $\gl_n$, $n>2$, and, more generally, to $\gl_{n|m}$ in future publications. The case of $\gl(1|1)$ is mentioned in \cite{FJM2}.

\medskip

We will return to the study of algebras $\mc A_N$ in \cite{FJM2}, where we will discuss, among other issues, the construction of $\mathcal{A}_N$ representations in terms of intertwining operators and properties of coproduct.

\medskip

The paper is constructed as follows. In Section \ref{gl1 sec} we recall the quantum toroidal algebra of type $\gl_1$, its simple properties and prove a few technical statements used later. In Section \ref{gl2 sec} we recall the quantum toroidal algebra of type $\gl_2$, define and study various fused currents. Section \ref{A sec} is devoted to defining the ``new realization" $\mc A_0$ and the identification of $\mc A_0$ and ${\E}_2[K]$. In Section \ref{shifted sec} we define and study the shifted ``new realizations" $\mc A_N$ and its representations. In Appendices, Sections \ref{app iota sec}, \ref{app tau sec}, and \ref{coproduct proof sec}, we prove our main Theorems \ref{tau thm}, \ref{tau thm}, and \ref{coproduct thm}.
 
\section{Quantum toroidal $\gl_1$}\label{gl1 sec}
We fix non-zero complex numbers
$q_1,q_2,q_3$ such that $q_1q_2q_3=1$.
We assume that $q_1,q_2$ are generic, meaning
that $q_1^aq_2^b=1$ with $a,b\in\Z$ if and only if $a=b=0$. 

We fix a choice of logarithms $\log q_i$ such that $\log q_1+\log q_2+\log q_3=0$ and define
$q_i^\alpha=e^{\alpha\log q_i}$ for all complex numbers $\alpha$. 

We set 
\begin{align*}
s_i&=q_i^{1/2}, \hspace{115pt} i=1,2,3,\\
\kappa_r&=(1-q_1^r)(1-q_2^r)(1-q_3^r), \qquad  r\in \Z.
\end{align*}

\subsection{Definition and basic properties}
We recall some well-known facts about quantum toroidal algebra associated 
with $\gl_1$. 

Let
\begin{align*}
& g(z,w)=(z-q_1w)(z-q_2w)(z-q_3w).
\end{align*} 

The quantum toroidal $\gl_1$ algebra $\E_1=\E_1(q_1,q_2,q_3)$ is a unital associative algebra 
generated by elements $e_k,f_k, h_r$,
where $k\in\Z$, $r\in\Z\backslash\{0\}$,
and invertible central elements $C, \psi_0$.
In terms of the generating series
\begin{align*}
&e(z) =\sum_{k\in \Z} e_{k}z^{-k}, \quad 
f(z) =\sum_{k\in\Z} f_{k}z^{-k}, \quad 
\psi^{\pm}(z) = 
\psi_0^{\pm1}\exp\bigl(\sum_{r=1}^\infty \kappa_r h_{\pm r}z^{\mp r}\bigr)\,,
\end{align*}
the defining relations read as follows:
\begin{align*}
&[\psi^\pm(z),\psi^\pm(w)]=0\,,\\
&\frac{g(z,C^{-1}w)}{g(z,Cw)}\, \psi^+(z)\psi^-(w)= \frac{g(C^{-1}w,z)}{g(Cw,z)} \,\psi^-(w)\psi^+(z) \,,
\\ 
&g(C^{(1\pm1)/2}z,w)\psi^\pm(z)e(w)+g(w,C^{(1\pm1)/2}z)e(w)\psi^\pm(z)=0
\,,\\ 
&g(w,C^{(1\mp1)/2}z) \psi^\pm(z)f(w)+g(C^{(1\mp1)/2}z,w)f(w)\psi^\pm(z)=0\,,
\\
&[e(z),f(w)]=\frac{1}{\kappa_1}
(\delta\bigl(\frac{Cw}{z}\bigr)\psi^+(w)
-\delta\bigl(\frac{Cz}{w}\bigr)\psi^-(z)),\\
&g(z,w)e(z)e(w)+g(w,z)e(w)e(z)=0, \\
&g(w,z)f(z)f(w)+g(z,w)f(w)f(z)=0,\\
&\mathop{\on{Sym}}_{z_1,z_2,z_3}\, z_1 z_2^2 [e(z_1),[e(z_2),e(z_3)]]=0,\\
&\mathop{\on{Sym}}_{z_1,z_2,z_3}\, z_1 z_2^2 [f(z_1),[f(z_2),f(z_3)]]=0\,,
\end{align*}
where $\Sym$ stands for symmetrization in the variables indicated.

We note that the first two relations are equivalent to the following relations for Cartan elements $h_r$,
\begin{align}\label{hh commutator}
[h_{r},h_{s}]=
\delta_{r+s,0}\,\frac{1}{r} \frac{C^{r}-C^{-r}}{\kappa_r}\,,\qquad r\in\Z\backslash\{0\}.
\end{align}
The last two relations are called Serre relations and can be equivalently written as
\begin{align*}
[e_n,[e_{n-1},e_{n+1}]]=0\,,\quad 
[f_n,[f_{n-1},f_{n+1}]]=0\,, \quad n\in\Z.
\end{align*}

Algebra $\E_1$ has a topological Hopf algebra structure. The coproduct $\Delta_1:\ \E_1\to \E_1 \hat{\otimes} \E_1$ given by
\begin{align}
&\Delta_1\psi^-(z)=\psi^-(z)\otimes\psi^-(C_1z)\,,\nn\\
&\Delta_1\psi^+(z)=\psi^+(C_2z)\otimes\psi^+(z)\,,\nn\\
&\Delta_1 e(z)=e(z)\otimes 1+\psi^-(z)\otimes e(C_1z)\,,\label{E1 coproduct}\\ 
&\Delta_1 f(z)=f(C_2z)\otimes \psi^+(z)+1\otimes f(z)\,,\nn\\ 
&\Delta_1 C=C\otimes C, \nn
\end{align}
where $C_1=C\otimes 1$, $C_2=1\otimes C$. The antipode $S_1:\ \E_1\to \E_1$ 
is given by
\begin{align}
&S_1\psi^\pm(z)=\psi^\pm(C^{-1}z)^{-1}, \nn \\
&S_1e(z)=-\psi^-(C^{-1}z)^{-1}e(C^{-1}z), \label{antipode} \\ 
&S_1f(z)=-f(C^{-1}z) \psi^+(C^{-1}z)^{-1},\nn\\
&S_1\, C=C^{-1}. \nn
\end{align}
The counit map $\epsilon:\ \E_1\to\C$ is given by $\epsilon\, e(z) = \epsilon\, f(z) = 0$, $\epsilon \,\psi^\pm (z) = 1$, $\epsilon\, C = 1$.

\medskip

Algebra $\E_1$  has a homogeneous $\Z$-grading which we call degree and denote by $\deg$. The degrees of elements are defined  by the assignment
\begin{align}\label{hom grading}
&\deg e_k=\deg f_k=k\,,\quad \deg h_r=r\,,\quad 
\deg C=\deg \psi_0=0\,.
\end{align}

Algebra $\E_1(q_1,q_2,q_3)$ does not depend on the order of parameters $q_i$. In other words for any permutation $\sigma\in S_3$ we have a Hopf algebra isomorphism
$\E_1(q_1,q_2,q_3)\xrightarrow{\sim} \E_1(q_{\sigma(1)},q_{\sigma(2)},q_{\sigma(3)})$  mapping each generator to itself.

For any $a\in\C^\times$, there exists a Hopf algebra automorphism of $\E_1$ called the shift of spectral parameter, mapping $C\mapsto C$ and 
\begin{align}\label{shift iso}
    e(z)\mapsto e(az), \qquad   f(z)\mapsto f(az),\qquad  \psi^\pm(z)\mapsto \psi^\pm(az).
\end{align}


\subsection{Representations.}
Algebra $\E_1$ has a rich representation theory which is extensively studied but still not fully
understood. 
Highest weight modules were studied in \cite{FFJMM, FJMM1}. Many interesting modules are obtained by twisting highest weight modules with automorphisms of $\E_1$. In this paper, we explicitly use only vector representations and Fock modules in a bosonic realization which we now remind.

We say an  $\E_1$ module $\mc M$ has level $\ell\in\C^\times$ if the central element $C$ acts in $\mc M$ by scalar $\ell$.

Set
\begin{align}\label{om}
\omega(x)=\frac{(1- q_2x)(1- q_3x)}{(1-x)(1-q_2q_3x)}\,.
\end{align}

We note the identities
\begin{align*}
\omega(q_1x)=\omega(x^{-1})\,, \qquad 
\frac{\omega(q_1 z/w)}{\omega(z/w)}=-\frac{g(z,w)}{g(w,z)}\,.
\end{align*}

For a rational function $p(z)$, denote  a Laurent series in $z$ given by the expansion of $p(z)$ at $z=0$ by $p^+(z)$. Similarly,  denote a Laurent series in $z^{-1}$ given by the expansion of $p(z)$ at $z=\infty$  by $p^-(z)$. For example, if $p(z)=(1-z)^{-1}$, then  $p^+(z)-p^-(z)=\delta(z)$.
 
The vector representation $V_1(v)$ of color $1$ and parameter $v$  is a level $1$ $\E_1$ module with
a basis $\{\ket{i,v}\}_{i\in\Z}$. The action  is given by 
\begin{align}
&(s_1^{-1}-s_1)e(z)\ket{i,v}=\delta(q_1^iv/z)\,\ket{i+1,v},\nn\\
&(s_1^{-1}-s_1)f(z)\ket{i,v}=\delta(q_1^{i-1}v/z)\,\ket{i-1,v},\label{vector rep}\\
&\psi^\pm(z)\ket{i,v}=\omega^\pm(q_1^iv/z)\,\ket{i,v}\,,
\nn
\end{align}
and $C=1$, $\psi_0=1$.

We remark that there is an isomorphism of modules $V_1(v)\xrightarrow{\sim} V_1(q_1v)$ mapping 
$\ket{i,v}\mapsto \ket{i-1,q_1v}$.

We use the anti-automorphism of $\E_1$ given by $e(z)\leftrightarrow f(z)$, 
$\psi^\pm(z)\leftrightarrow \psi^\pm(z)$,  to 
view $V_1(v)$ as a right module. We denote this right module by $V^R_1(v)$.

Vector representations $V_i(v)$, $V_i^R(v)$ of color $i\in\{1,2,3\}$ and parameter $v$ are defined by the same formulas with $q_1$ exchanged with $q_i$. 

\medskip

The Fock representation $\F_2(v)$ of color $2$ and parameter $v$
is an irreducible $\E_1$ module of level $s_2$ given as follows. As vector space we have 
$\F_2(v)=\C[h_r]_{r<0}$.
The central elements act by scalars $C=s_2$, $\psi_0=1$. The elements $h_r$, $r<0$, act by multiplication operators. The elements $h_r$, $r>0$, act by differential operators $-(rs_2^r(1-q_1^r)(1-q_3^r))^{-1}\partial_{h_{-r}}$. The series $e(z)$ and $f(z)$ act by vertex operators:
\begin{align*}
&e(z)=\frac{s_2^{-1}v} {(1-q_1)(1-q_3)} 
\,
\exp\Bigl(\sum_{r>0}(1-q_1^r)(1-q_3^r)h_{-r}z^r\Bigr)
\exp\Bigl(\sum_{r>0}(1-q_1^{-r})(1-q_3^{-r})s_2^{-r} h_{r}z^{-r}\Bigr)\,,
\\
& f(z)=\frac{-s_2v^{-1}}{ {(1-q_1)(1-q_3)}} 
\,
\exp\Bigl(-\sum_{r>0}(1-q_1^{r})(1-q_3^{r}) s_2^{r} h_{-r}z^r\Bigr)
\exp\Bigl(-\sum_{r>0}(1-q_1^{-r})(1-q_3^{-r}) h_{r}z^{-r}\Bigr)\,.
\end{align*}
Fock representations $\F_i(v)$ of color $i\in\{1,2,3\}$ and parameter $v$ are defined by the same formulas with $q_2$ exchanged with $q_i$. 

\medskip
We have the following ``level 1" property.
\begin{lem}\label{level 1 ideal}
In $\F_i(v)$ we have
\begin{align*}
e(z) e(q_iz)=0, \qquad  f(q_i z) f(z)=0.
\end{align*}
\end{lem}
\begin{proof}
It is a straightforward computation with vertex operators using \eqref{hh commutator}.
\end{proof}
\subsection{Fused currents} We recall the definition of fused currents, see \cite{FJMM3}, and define several of them which will be used later.

Fused currents are not defined on algebra level, these currents are defined only in the sense of vertex operator algebras. Namely, the coefficients of fused currents are well defined operators in admissible modules.

Quite generally, let $a(z)=\sum_{k\in\Z} a_kz^{-k}$ and  $b(z)=\sum_{k\in\Z} b_k z^{-k}$ be two currents acting on a graded vector space $\mc M$. Assume that the action is graded, $\deg a_k=\deg b_k=k$ for all $k$, and ``admissible", meaning that for any vector $w\in \mc M$, $a_k\,w=b_k\,w=0$ for large enough $k$. Assume that we have a commutation relation
$$
g_1(z,w) a(z) b(w)=g_2(w,z) b(w) a(z),
$$
where $g_1(z,w), g_2(w,z)\in \C[z,w]$ are homogeneous polynomials of the same degree.
Then the condition of admissibility implies that all matrix coefficients of the current $g_1(z,w) a(z) b(w)$ are Laurent polynomials in $z, w$. 

Let $\alpha$ be a constant such that $g_1(z, \alpha z)\neq 0$. Then $a(z)b(\alpha z)$ is a well defined operator on $\mc M$. We call this operator a fused current. 

Note that if, in addition, $g_2(\alpha z,z)\neq 0$ then  $b(\alpha z)a(z)$ is also a fused current and $a(z)b(\alpha z)=c b(\alpha z)a(z) $  where the constant $c$  is given by $c=g_2(\alpha z,z)/g_1(z,\alpha z)$. If $g_2(\alpha z,z)=0$ and $a(z)b(\alpha z)\neq 0$,  
then $b(\alpha z)a(z)$ is not well defined.

In general the fusion product is not associative. However, we have the following lemma.

\begin{lem}\label{ass lem}
Let graded currents $a_i(z)$ satisfy $g_{ij}(z,w)a_i(z)a_j(w)= g_{ji}(w,z)a_j(w)a_i(z)$, $i,j\in\{1,2,3\}$,  where $g_{ij}(z,w)$ are homogeneous polynomials. Assume that $g_{ij}(\al_iz,\al_jz)\neq 0$ for $i<j$.  Then we have equality of well defined fused currents
\begin{align*}
    \Big(a_1(\al_1 z)a_2(\al_2 z_1)\Big) a_3(\al_3 z)=a_1(\al_1 z)\Big(a_2(\al_2 z_1)a_3(\al_3 z)\Big) .
\end{align*}
\qed
\end{lem}
Due to Lemma \ref{ass lem}, all fusion products in this paper are associative  and we never write parenthesis to specify the order the fusion product is taken.

\medskip 
Now let us return to $\E_1$ modules.
An $\E_1$ module $\mc M$ is 
said to be admissible if it is graded with respect to degree \eqref{hom grading}, 
$\mc M=\oplus_{d\in\Z}\mc M_d$, all graded components are finite-dimensional, $\dim \mc M_d<\infty$, and there exists $d_0\in\Z$ such that $\mc M_d=0$ for all $d>d_0$. 

The Fock modules $\F_i(v)$ are admissible and vector representations $V_i(v)$ are not.

Let $\mc{M}_1$  be an admissible $\E_1$ module of level $C$.  Consider all elements of $\E_1$ and Fourier coefficients of the fused currents as elements of  $\End \mc{M}_1$.

As a warm up, we have the following relations.
\begin{lem}\label{zero currents lem}
The following identities hold in $\mc{M}_1$ 
\begin{align*}
&\psi^+(C^{-1}z)e(q_iz)=0\,,\qquad e(z)\psi^-(q_iz)=0\,,\\ 
& f(q_iz)\psi^-(C^{-1}z)=0\,,\qquad \psi^+(q_iz)f(z)=0\,,\\\
&\qquad e(z)^2=0, \qquad \qquad\qquad  f(z)^2=0.
\end{align*}
Here $i\in\{1,2,3\}$.
\end{lem}
\begin{proof}
On an admissible module, each matrix element of $\psi^-(w)e(z)$ is a 
Laurent polynomial of $z,w$. Hence one can substitute $w=q_iz$ in 
the identity of rational functions
\begin{align*}
e(z) \psi^-(w)=\psi^-(w)e(z)\prod_{i=1}^3\frac{w-q_iz}{w-q_i^{-1}z}
\end{align*}
to obtain $e(z)\psi^-(q_iz)=0$. The other identities can be shown similarly.
\end{proof}

The following fused
currents play an important role in our studies.
For $r\in\Z_{>0}$ and $i\in\{1,2,3\}$, let 
\begin{align}
&e^{(r)}_{q_i}(z)=c_i^re(z)\,e(q_iz)\,\cdots\, e(q_i^{r-1}z)\,,\label{e}\\ 
&f^{(r)}_{q_i}(z)=(-c_i)^rf(q_i^{r-1}z)\,\cdots\, f(q_iz)\,f(z)\,,\label{f}
\end{align}
where $c_i$ are constants given by
\begin{align*}
c_i=\frac{\kappa_1}{s_i-s_i^{-1}}\,.
\end{align*}

We will also use currents
\begin{align*}
\psi_{q_{i}}^{\pm,(r)}(z)=\psi^\pm(z)\psi^\pm(q_iz)\dots \psi^\pm(q_i^{r-1}z).
\end{align*}
We set $e^{(0)}_{q_i}(z)=f^{(0)}_{q_i}(z)=\psi^{\pm,(0)}_{q_i}(z)=1$.

Let $\E_1[K]=\E_1\otimes\C[K^{\pm1}]$ be the extension of $\E_1$ by a split central invertible group-like element $K$.

Using algebra $\E_1[K]$, we define fused currents $k_r^\pm(z)$, $r\in \Z_{\geq 0}$, 
modifying fused currents $e^{(r)}_{q_1}(z)$, $f^{(r)}_{q_1}(z)$ by Cartan currents of $\mc E_1,$ as follows.

Let $k_0^\pm(z)$ be series in $z^{\mp1}$ which are solutions of difference equations
\begin{align}\label{k0}
k_0^\pm(q_1z)=\psi_0^{\mp1}\psi^\pm(z)k_0^\pm(z),  
\qquad k_0^+(\infty)=K
,
\quad  k_0^-(0)=K^{-1}.
\end{align}
Then explicitly
\begin{align}\label{k0 expl}
k_0^\pm(z)&=K^{\pm1} \exp\Bigl(\sum_{r>0}\frac{\kappa_r}{q_1^{\mp r}-1}
h_{\pm r}z^{\mp r}\Bigr)\,.
\end{align}
Clearly, the subalgebra of $\E_1[K]$ generated by $k_0^\pm(z)$ and $\psi_0$ coincides with the Cartan subalgebra generated by $\psi^\pm(z)$ and $K$.

For $r\in\Z_{\geq 0}$, define
\begin{align}\label{k}
k^+_r(z) &=f^{(r)}_{q_1}(z)k_0^+(z)\,,\quad 
k^-_r(z)=k_0^-(z)e^{(r)}_{q_1}(z)\,.
\end{align}

We choose this notation, since later the currents $k^\pm_r(z)$ will play a role of Cartan currents for various quantum toroidal algebras associated with
$\gl_2$. Note that $k_r^\pm(z)$, $r\neq 0$, are not one-sided currents, the plus and minus indicate that we used $\psi^\pm(z)$ in the construction.
Note also that here we chose $q_1$ over $q_2$ and $q_3$, breaking the symmetry. This is not reflected in our notation and we hope it does not lead to a confusion.

\subsection{Properties of currents $k_r^\pm(z)$}
We start with the action on the Fock space.

\begin{lem}\label{k on fock lem}
    In $\mc F_1(u)$ the operators $k^\pm_r(z)$ with $r\geq 2$ act by zero. In addition, we have
    \begin{align}
        k_0^\pm(s_1z)+k_1^\mp(z)=0.
    \end{align}
\end{lem}
\begin{proof}
 The first part follows from Lemma \ref{level 1 ideal}. The last equation is checked explicitly.
\end{proof}

We compute the coproduct of $k_r^\pm(z)$.
\begin{lem}\label{lem:DeltaK}
The coproduct of the currents $k_r^\pm(z)$ 
is given by
\begin{align*}
&\Delta k^+_r(z)=\sum_{r_1,r_2\ge0\atop r_1+r_2=r}
k^+_{r_1}(C_2z)\otimes \psi_0^{r_1} k^+_{r_2}(q_1^{r_1}z)\,,\\
&\Delta k^-_r(z)=\sum_{r_1,r_2\ge0\atop r_1+r_2=r}
 \psi_0^{-r_2} k^-_{r_1}(q_1^{r_2}z)\otimes k^-_{r_2}(C_1z)\,.
\end{align*}
\end{lem}
\begin{proof}
The case $r=0$ is obvious. 
We show the lemma by induction on $r$ using the recursion
\begin{align*}
k^+_{r+1}(z)=-c_1f(q_1^rz)k^+_r(z)\,, \quad
k^-_{r+1}(z)=k^-_r(z)c_1 e(q_1^rz)\,.
\end{align*} 
We have
\begin{align*}
\Delta k^+_{r+1}(z)&=(-c_1)\bigl(f(C_2q_1^rz)\otimes\psi^+(q_1^rz)
+1\otimes f(q_1^rz)\bigr) 
\sum_{r_1+r_2=r\atop r_1,r_2\ge0}k^+_{r_1}(C_2z)\otimes \psi_0^{r_1}k^+_{r_2}(q_1^{r_1}z)\\
&=\sum_{r_1+r_2=r\atop r_1,r_2\ge0}
k^+_{r_1+1}(C_2z)\otimes\psi_0^{r_1}\psi^+(q_1^rz)k^+_{r_2}(q_1^{r_1}z)
+\sum_{r_1+r_2=r\atop r_1,r_2\ge0}k^+_{r_1}(C_2z)\otimes \psi_0^{r_1}k^+_{r_2+1}(q_1^{r_1}z)\,.
\end{align*}
By Lemma \ref{zero currents lem}, $\psi^+(q_1^{r}z)f(q_1^{r-1}z)=0$, and
only the term with $r_2=0$ survives in the first sum.
The assertion follows by noting that 
$\psi^+(q_1^rz)k^+_0(q_1^rz)=\psi_0k_0^+(q_1^{r+1}z)$.

The case of $k^-_r(z)$ is entirely similar.

\end{proof}

Next, we need some identities involving $k_r^{\pm}(z)$.

\begin{lem}\label{k identity lemma}
The following identities hold in $\End \mc M_1$:
\begin{align*}
e(u)k^-_r(q_1^{-r}z)&-\omega^-_1
\Bigl(\frac{z}{u}\Bigr)k^-_r(q_1^{-r}z)e(u)
\\
&=\frac{c_1}{\kappa_1}\Bigl(
\delta\Bigl(\frac{q_1^{r+1}u}{z}\Bigr)
\psi_0
\psi^-(u)k^-_{r+1}(q_1^{-r-1}z)-
\delta\Bigl(\frac{u}{z}\Bigr)k^-_{r+1}(q_1^{-r}z)
\Bigr) \,,
\\
e(u)k^+_r(q_1^{-r}w)&-\omega^-_1\Bigl(\frac{Cq_1^{-r}w}{u}\Bigr)
k^+_r(q_1^{-r}w)e(u)
\\
&=\frac{c_1}{\kappa_1}\Bigl(
\delta\Bigl(\frac{Cq_1u}{w}\Bigr)
\psi^-(u)k^+_{r-1}(q_1^{-r}w)
-\delta\Bigl(\frac{Cq_1^{-r}w}{u}\Bigr)
\psi_0
k^+_{r-1}(q_1^{-r+1}w)
\Bigr) \,,
\\
k^-_r(q_1^{-r}z)f(u)&-\omega^+_1\Bigl(\frac{Cq_1^{-r}z}{u}\Bigr)
f(u)k^-_r(q_1^{-r}z)
\\
&=\frac{c_1}{\kappa_1}\Bigl(
\delta\Bigl(\frac{Cq_1u}{z}\Bigr)k^-_{r-1}(q_1^{-r}z)
\psi^+(u)-
\delta\Bigl(\frac{Cq_1^{-r}z}{u}\Bigr)
\psi_0^{-1}
k^-_{r-1}(q_1^{-r+1}z)
\Bigr) \,,
\\
k^+_r(q_1^{-r}w)f(u)&-\omega^+_1
\Bigl(\frac{w}{u}\Bigr)
f(u)k^+_r(q_1^{-r}w)
\\
&=\frac{c_1}{\kappa_1}\Bigl(
\delta\Bigl(\frac{q_1^{r+1}u}{w}\Bigr)k^+_{r+1}(q_1^{-r-1}w)
\psi_0^{-1}
\psi^+(u)
-\delta\Bigl(\frac{u}{w}\Bigr)k^+_{r+1}(q_1^{-r}w)
\Bigr) \,.
\end{align*}
\end{lem}
\begin{proof}
We show the first equation. The other cases are similar. 

We use induction on $r$.
For $r=0$, we compute
\begin{align*}
&e(u)k^-_0(z)-\omega^-_1\Bigl(\frac{z}{u}\Bigr)k_0^-(z)e(u)=\Bigl(\omega^+\Bigl(\frac{z}{u}\Bigr)-\omega^-\Bigl(\frac{z}{u}\Bigr)
\Bigr) k^-_0(z)e(u)\,
\\
&=\frac{c_1}{s_1-s_1^{-1}}
\Bigl(\delta\Bigl(\frac{q_1^{-1}z}{u}\Bigr)
-\delta\Bigl(\frac{z}{u}\Bigr)\Bigr)k^-_0(z)e(u)\,
=\frac{c_1}{s_1-s_1^{-1}}
\Bigl(\delta\Bigl(\frac{q_1^{-1}z}{u}\Bigr)
k^-_0(z)e(q_1^{-1}z)
-\delta\Bigl(\frac{z}{u}\Bigr)k^-_0(z)e(z)
\Bigr)
\,
\\
&=\frac{c_1}{\kappa_1}
\Bigl(\delta\Bigl(\frac{q_1^{-1}z}{u}\Bigr)
{\psi_0}\psi^-(q_1^{-1}z)
k_1^-(q_1^{-1}z)
-\delta\Bigl(\frac{z}{u}\Bigr)k^-_1(z)
\Bigr)\,,
\end{align*}
where in the last step we used \eqref{k0}.

Suppose the first relation of the lemma holds for $r$. 
Using the induction hypothesis  together with  the recursion $k^-_{r+1}(q_1^{-r-1}z)=c_1k^-_r(q_1^{-r-1}z)e(q_1^{-1}z)$
, we have
\begin{align}
&e(u)k^-_{r+1}(q_1^{-r-1}z)
-\omega^-\Bigl(\frac{z}{u}\Bigr)k^-_{r+1}(q_1^{-r-1}z)e(u)\notag
\\
&=c_1k^-_r(q_1^{-r-1}z)
\Bigl\{
\omega^-\Bigl(q_1^{-1}\frac{z}{u}\Bigr)e(u)e(q_1^{-1}z)
-
\omega^-\Bigl(\frac{z}{u}\Bigr)e(q_1^{-1}z)e(u)
\Bigr\}\notag
\\
&+\frac{c_1^2}{\kappa_1}
\Bigl\{
\delta\Bigl(\frac{q_1^{r+2}u}{z}\Bigr)\psi_0\psi^-(u)k^-_{r+1}(q_1^{-r-2}z)
-\delta\Bigl(\frac{q_1u}{z}\Bigr)k^-_{r+1}(q_1^{-r-1}z)
\Bigr\}
e(q_1^{-1}z)\,. \label{aux eq}
\end{align}
Consider the first summand.
From the relations between $e(z)$ and $e(w)$, we see that
\begin{align*}
(u-z)(u-q_1^{-1}z)
(u-q_1^{-2}z)\Big(\omega\Bigl(\frac{q_1^{-1}z}{u}\Bigr) e(u)e(q_1^{-1}z)
-\omega\Bigl(\frac{z}{u}\Bigr)e(q_1^{-1}z)e(u)\Big)=0\,.
\end{align*}

This yields 
\begin{align*}
\omega^+\Bigl(\frac{q_1^{-1}z}{u}\Bigr) e(u)e(q_1^{-1}z)
-\omega^-\Bigl(\frac{z}{u}\Bigr)&e(q_1^{-1}z)e(u)
=\sum_{s=0}^2\delta\Bigl(\frac{q_1^{-s}z}{u}\Bigr)
\res_{u=q_1^{-s}z}\omega\Bigl(\frac{q_1^{-1}z}{u}\Bigr) e(u)e(q_1^{-1}z)
\frac{du}{u}\,
\\
&=\frac{c_1}{s_1-s_1^{-1}}\Bigl\{
-\delta\Bigl(\frac{z}{u}\Bigr)e(q_1^{-1}z)e(z)
+\delta\Bigl(q_1^{-2}\frac{z}{u}\Bigr)e(q_1^{-2}z)e(q_1^{-1}z)
\Bigr\},
\end{align*}
where we used $e(z)^2=0$, see Lemma \ref{zero currents lem}. This can be 
further rewritten as
\begin{align*}
\omega^-\Bigl(\frac{q_1^{-1}z}{u}\Bigr) e(u)e(q_1^{-1}z)
-\omega^-\Bigl(\frac{z}{u}\Bigr)e(q_1^{-1}z)e(u)
=-\frac{c_1}{s_1-s_1^{-1}}
\delta\Bigl(\frac{z}{u}\Bigr)e(q_1^{-1}z)e(z)\,.
\end{align*}

It follows that the first term in \eqref{aux eq}
equals
\begin{align*}
c_1k^-_r(q_1^{-r-1}z)\cdot (-1)\frac{c_1^2}{\kappa_1}
\delta\Bigl(\frac{z}{u}\Bigr) 
e(q_1^{-1}z)e(z)
=-\frac{c_1}{\kappa_1}\delta\Bigl(\frac{z}{u}\Bigr) k^-_{r+2}(q_1^{-r-1}z).
\end{align*}
Note again that  $k_{r+1}^-(q_1^{-r-1}z)e(q_1^{-1}z)=k_{r}^-(q_1^{-r-1}z)e(q_1^{-1}z)^2=0$. Therefore the second term in \eqref{aux eq} becomes
\begin{align*}
\frac{c_1}{\kappa_1}\delta\Bigl(\frac{q_1^{r+2}z}{u}\Bigr) 
\psi_0\psi^-(u)k^-_{r+2}(q_1^{-r-2}z)\,.
\end{align*}
This completes the proof.
\end{proof}

\section{Quantum toroidal $\gl_2$}\label{gl2 sec}
We fix non-zero complex numbers
$\bar q_1,\bar q_2,\bar q_3$ such that $\bar q_1\bar q_2\bar q_3=1$.
We assume the numbers $\bar q_1,\bar q_2$ are generic, meaning
that $\bar q_1^a\bar q_2^b=1$ with $a,b\in\Z$ if and only if $a=b=0$. 

We fix a choice of logarithms $\log \bar q_i$ such that $\log \bar q_1+\log \bar q_2+\log \bar q_3=0$ and define
$\bar q_i^\alpha=e^{\alpha\log \bar q_i}$ for all complex numbers $\alpha$. 

We set 
\begin{align*}
\bar s_i&=\bar q_i^{1/2}, \hspace{115pt} i=1,2,3,\\
\bar \kappa_r&=(1-\bar q_1^r)(1-\bar q_2^r)(1-\bar q_3^r), \qquad  r\in \Z.
\end{align*}

\subsection{Definition and basic properties} We recall the well-known facts about quantum toroidal algebra associated with 
$\gl_2$.

Let
\begin{align*}
& g_{00}(z,w)=g_{11}(z,w)=(z-\bar q_2w), \\
& g_{01}(z,w)=g_{10}(z,w)=(z-\bar q_1w)(z-\bar q_3w).
\end{align*} 
In the case of $\gl_2$ the exchange functions have the symmetry $g_{ij}(z,w)=g_{ji}(z,w)$. 

The quantum toroidal $\gl_2$ algebra $\E_2=\E_2(\bar q_1,\bar q_2,\bar q_3)$ is a unital associative algebra 
generated by elements $E_{i,k},F_{i,k},H_{i,r}$, invertible elements $\Psi_{i,0}$,
where $i\in\Z/2\Z$, $k\in\Z$, $r\in\Z\backslash\{0\}$, and
an invertible central element $C$.
In terms of the generating series
\begin{align*}
&E_i(z) =\sum_{k\in \Z} E_{i,k}z^{-k}, \quad 
F_i(z) =\sum_{k\in\Z} F_{i,k}z^{-k}, \quad 
\Psi_i^{\pm}(z) = \Psi_{i,0}^{\pm1}
\exp\bigl(\pm(\bar s_2-\bar s_2^{-1})
\sum_{r=1}^\infty  H_{i,\pm r}z^{\mp r}\bigr)\,,
\end{align*}
the defining relations read as follows:
\begin{align*}
&[\Psi_i^\pm(z),\Psi_j^\pm(w)]=0\,,\\
&\frac{g_{ij}(z,C^{-1}w)}{g_{ij}(z,Cw)}\, \Psi_i^+(z)\Psi_j^-(w)= \frac{g_{j,i}(C^{-1}w,z)}{g_{ji}(Cw,z)} \,\Psi_j^-(w)\Psi_i^+(w) \,,
\\ 
&g_{ij}(C^{(1\pm1)/2}z,w)\Psi_i^\pm(z)E_j(w)+(-1)^{i+j}g_{ji}(w,C^{(1\pm1)/2}z)E_j(w)\Psi_i^\pm(z)=0
\,,\\ 
&g_{ji}(w,C^{(1\mp1)/2}z) \Psi_i^\pm(z)F_j(w)+(-1)^{i+j}g_{ij}(C^{(1\mp1)/2}z,w)F_j(w)\Psi_i^\pm(z)=0\,,
\\
&[E_i(z),F_j(w)]=\frac{\delta_{ij}}{\bar s_2-\bar s_2^{-1}} 
(\delta\bigl(\frac{Cw}{z}\bigr)\Psi_i^+(w)
-\delta\bigl(\frac{Cz}{w}\bigr)\Psi_i^-(z)),\\
&g_{ij}(z,w)E_i(z)E_j(w)+(-1)^{i+j}g_{ji}(w,z)E_j(w)E_i(z)
=0, \\
&g_{ji}(w,z)F_i(z)F_j(w)+(-1)^{i+j}g_{ij}(z,w)F_j(w)F_i(z)
=0,\\
&\mathop{\on{Sym}}_{z_1,z_2,z_3}\, [E_i(z_1),[E_i(z_2),[E_i(z_3),E_j(w)]_{\bar q_2}]]_{\bar q_2^{-1}}=0, \qquad i\neq j,\\
&\mathop{\on{Sym}}_{z_1,z_2,z_3}\, [F_i(z_1),[F_i(z_2),[F_i(z_3),F_j(w)]_{\bar q_2}]]_{\bar q_2^{-1}}=0, \qquad i\neq j,
\end{align*}
where $[A,B]_p=AB-pBA$.

There exists a topological coproduct $\Delta_2:\ \E_2\to \E_2 \hat{\otimes} \E_2$ given by
\begin{align*}
&\Delta_2 E_i(z)=E_i(z)\otimes 1+\Psi_i^-(z)\otimes E_i(C_1z)\,,\\ 
&\Delta_2\Psi_i^-(z)=\Psi_i^-(z)\otimes\Psi_i^-(C_1z)\,,\\
&\Delta_2 F_i(z)=F_i(C_2z)\otimes \Psi_i^+(z)+1\otimes F_i(z)\,,\\ 
&\Delta_2\Psi_i^+(z)=\Psi_i^+(C_2z)\otimes\Psi_i^+(z)\,,\\
&\Delta_2 x=x\otimes x\quad \text{for $x=C,\Psi_{i,0}$}\,,
\end{align*}
where $C_1=C\otimes 1$, $C_2=1\otimes C$, and $i=0,1$. 

There is also an antipode and a counit but we do not use them in this paper.

\medskip

The element $\Psi_0=\Psi_{0,0}\Psi_{1,0}$ is central.

Algebra $\E_2$  has a homogeneous $\Z$-grading which we call degree and denote by $\deg$. The degrees of elements are defined  by the assignment
\begin{align}\label{hom grading 2}
&\deg E_{i,k}=\deg F_{i,k}=k\,,\quad \deg H_{i,r}=r\,,\quad 
\deg C=\deg \Psi_{i,0}=0\,.
\end{align}

Algebra $\E_2$  has another $\Z$-grading  which we call weight and denote by $\wt$. The weights of elements are defined by the assignment
\begin{align}\label{weight 2}
&\wt E_{i,k}=(-1)^{i+1}, \quad   \wt F_{i,k}=(-1)^{i}\,,\quad \wt H_{i,r}= 
\wt C=\wt \Psi_{i,0}=0\,.
\end{align}

Algebra $\E_2(\bar q_1,\bar q_2,\bar q_3)$ is symmetric with respect to change of $\bar q_1\leftrightarrow \bar q_3$. Namely, we have a Hopf algebra isomorphism $\E_2(\bar q_1,\bar q_2,\bar q_3)\xrightarrow{\sim}\E_2(\bar q_3,\bar q_2,\bar q_1)$  mapping each generator to itself.

There is a diagram automorphism of $\E_2$ sending currents $E_i(z), F_i(z),\Psi_i(z),C$ to currents 
$E_{i+1}(z)$, $F_{i+1}(z)$, $\Psi_{i+1}(z),C$, respectively.

The subalgebra of $\E_2$ generated by $E(z)=E_1(z)$, $F(z)=F_1(z)$, $\Psi_i^\pm (z)$, $C$ is isomorphic to quantum affine algebra $U_{s_2}\widehat{\gl}_2$.


\subsection{Fused currents in $\E_2$.} We define some fused currents in $\E_2$. In particular, we follow \cite{FJMM3} and use fused currents to describe two commuting subalgebras $\E_1\simeq\E_1(\bar q_1\bar q_3^{-1},\bar q_2,\bar q_3^2)$ and $\check{ \E}_1\simeq\E_1(\bar q_1^{-1}\bar q_3,\bar q_2,\bar q_1^2)$  of $\E_2=\E_2(\bar q_1,\bar q_2,\bar q_3)$.

We denote the parameters of the algebras $\E_1$ and $\check{\E}_1$ as follows:
\begin{align}\label{q-bar q}
(q_1, q_2, q_3)=(\bar q_1\bar q_3^{-1},\bar q_2,\bar q_3^2), \qquad (\check{q}_1, \check{q}_2 ,\check{q}_3)=(\bar q_1^{-1}\bar q_3,\bar q_2,\bar q_1^2).
\end{align}
Then $\check{s}_i=\check{q}_i^{1/2}$, $\check{\kappa}_r=(1-\check{q}_1^r)(1-\check{q}_2^r)(1-\check{q}_3^r)$,  $\check {c}_1=\check{\kappa}_1(\check{s}_1-\check{s}_1)^{-1}$.

Note 
\begin{align}
q_1 \check{q}_1 =1, \qquad  q_2=\check{q_2}=\bar q_2. 
\end{align}

\medskip
An $\E_2$ module $W$ is 
said to be graded if 
$W=\oplus_{d_1,d_2\in\Z} W_{d_1,d_2}$,  $\dim W_{d_1,d_2}<\infty$, and for any $x\in \E_2$ of degree $a$ and weight $b$, we have  $x W_{d_1,d_2}\subset W_{d_1+a,d_2+b}$.

An $\E_2$ module $W$ is 
said to be admissible if it is graded and  for each $d_2\in \Z$ there exists $d$ (depending on $d_2$) such that such that $W_{d_1,d_2}=0$ for all $d_1>d$. 

Let $\mc M_2$ be an admissible $\E_2$ module of level $C$. 
For example, $\mc{M}_2$ can be the direct sum of tensor products of highest weight Verma modules with respect to parallel generators of $\E_2$ obtained by Miki automorphism, see \cite{FJMM2}. Consider elements of $\E_2$ and Fourier coefficients of the fused currents as elements of $\on{End}\, \mc{M}_2$. 

Set
\begin{align*}
    a_{12}=-\frac{s_2-s_2^{-1}}{s_1-s_1^{-1}}.
\end{align*}
Let
\begin{align*}
    e(q_1^{-1}z)=a_{12}E_1(z)E_0(\bar q_3z), \qquad f(q_1^{-1}z)=-a_{12}F_0(\bar q_3z)F_1(z), \qquad \psi^\pm(q_1^{-1}z)=\Psi_0^\pm(\bar q_3z)\Psi_1^\pm(z).
\end{align*}
Note that here $e(z)$ and $f(z)$ are fused currents.
Similarly, let 
\begin{align*}
    \check{e}(\check{q}_1^{-1}z)=-a_{12}E_1(z)E_0(\bar q_1z), \qquad \check{f}(\check{q}_1^{-1}z)=a_{12} F_0(\bar q_1z)F_1(z), \qquad \check{\psi}^\pm(\check{q}_1^{-1}z)=\Psi_0^\pm(\bar q_1z)\Psi_1^\pm(z).
\end{align*}

\begin{prop}\label{gl1 sub prop} (\cite{FJMM3}) The  currents $e(z)$, $f(z)$, $\psi^\pm(z)$, and central element $C$ satisfy the relation of the quantum toroidal algebra $\E_1(\bar q_1\bar q_3^{-1},\bar q_2,\bar q_3^2)$.

Similarly, the currents $\check{e}(z)$, $\check{f}(z)$, $\check{\psi}^\pm(z)$, and central element $C$ satisfy the relations of the quantum toroidal algebra $\E_1(\bar q_1^{-1}\bar q_3,\bar q_2,\bar q_1^2)$. \qed
\end{prop}

We denote the algebra generated by $e(z)$, $f(z)$, $\psi^\pm(z)$, and $C$ by $\E_1$.
We denote  the algebra generated by $\check{e}(z)$, $\check{f}(z)$, $\check{\psi}^\pm(z)$, and $C$ by $\check{\E}_1$.

We warn the reader that the fusion products are infinite sums and, therefore, strictly speaking, Fourier components of $e(z)$, $f(z)$, etc., belong not to $\E_2$ but rather to a completion of $\E_2$ with respect to the homogeneous grading. Often, abusing the language one says that $\E_1$ and $\check{\E}_1$ are subalgebras of $\E_2$.

\begin{prop} \label{commute prop}(\cite{FJMM3}) 
The algebras $\E_1$ and $\check{\E}_1$ commute.  \qed
\end{prop}

We remark that the Serre relations  (the last two defining relation of $\E_2$) play an essential part in the proof of Proposition \ref{commute prop}.

We note that the central elements $C$ of algebras $\E_2$, $\E_1$, $\check{\E}_1$ are the same. In addition, we have $\psi_0=\check{\psi}_0=\Psi_0$.

\medskip 

As in \eqref{e},\eqref{f}, for $r\in\Z_{\geq 0}$, we set
\begin{align*}
e^{(r)}(q_1^{-1}z)&=e_{q_1}^{(r)}(q_1^{-1}z)= (c_1a_{12})^rE_1(z)E_0(\bar q_3z)E_1(\bar q_1^{\,-1}\bar q_3z)E_0(\bar q_1^{\,-1}\bar q_3^2z)E_1(\bar q_1^{\,-2}\bar q_3^2z)\dots E_0(\bar q_1^{1-r}\bar q_3^{r}z), \\
f^{(r)}(q_1^{-1}z)&=f_{q_1}^{(r)} (q_1^{-1}z)=(c_1a_{12})^r
F_0(\bar q_1^r\bar q_3^{1-r}z)  \dots  F_1(\bar q_1^{\,-2}\bar q_3^2z)F_0(\bar q_1^{\,-1}\bar q_3^2z)F_1(\bar q_1^{\,-1}\bar q_3z) F_0(\bar q_3z)F_1(z).
\end{align*}
and 
\begin{align*}
\check{e}^{(r)}(\check{q}_1^{-1}z)
&=\check{e}_{\check{q}_1}^{(r)}(z)=
(-\check c_1
a_{12})^rE_1(z)E_0(\bar q_1z)E_1(\bar q_1\bar q_3^{\,-1}z)E_0(\bar q_1^2\bar q_3^{\,-1}z)E_1(\bar q_1^2\bar q_3^{\,-2}z)\dots E_0(\bar q_1^{r}\bar q_3^{1-r}z), \\
\check{f}^{(r)}(\check{q}_1^{-1}z)
&=\check{f}_{\check{q}_1}^{(r)}(\check{q}_1^{-1}z) =
(-\check c_1
a_{12})^r
F_0(\bar q_1^{1-r}\bar q_3^rz)  \dots  F_1(\bar q_1^2\bar q_3^{\,-2}z)F_0(\bar q_1^2\bar q_3^{\,-1}z)F_1(\bar q_1\bar q_3^{\,-1}z) F_0(\bar q_1z)F_1(z) .
\end{align*} 
We will also use currents $\psi^{\pm,(r)}(z)=\psi^{\pm,(r)}_{q_1}(z)$ and $\check{\psi}^{\pm,(r)}(z)=\check{\psi}^{\pm,(r)}_{\check{q}_1}(z)$.

Next we define fused currents $X^\pm_ i(z)$, $i\in\Z$, by the formulas
\begin{align*}
     X^+_{-r}(z)&=E_1(q_1^{-r}z)\,e^{(r)}(q_1^{-r}z),\\
      X^+_{r}(z)&=f^{(r)}(C^{-1}z)\,\big(\psi^{+,(r)}(C^{-1}z)\big)^{-1}E_1(q_1^rz),\\ 
     X^-_{-r}(z)&=(s_2-s_2^{-1})F_1(q_1^rz)\,\big(\psi^{-,(r)}(C^{-1}z)\big)^{-1}e^{(r)}(C^{-1}z), \\
      X^-_{r}(z)&=-(s_2-s_2^{-1})f^{(r)}(q_1^{-r}z)\, F_1(q_1^{-r}z),
      \end{align*}
where $r\geq 0$.

The formulas for $X_1^+(z)$ and $X_{-1}^-(z)$ can be simplified.
\begin{lem}\label{X simpl} We have
\begin{align*}
X_1^+(Cz)
&=F_0(\bar q_1z)\Psi_0^+(\bar q_1z)^{-1}, \\
X_{-1}^-(Cz)
&=(s_2-s_2^{-1})\Psi_0^-(\bar q_1z)^{-1}E_0(\bar q_1 z).
\end{align*}
\end{lem}
\begin{proof}
From the $[E_1(bz),F_1(bz)]$ relation, for any $b\neq 0$, we have 
\begin{align*}
    (1-\frac{z}{w})F_1(bz)E_1(bCw)
\Psi_1^+(bz)^{-1}\Big|_{w=z}=\frac{1}{s_2-s_2^{-1}}\,.
\end{align*}
Then we compute
\begin{align*}
    X_1^+(Cz)&=c_1a_{12} F_0(\bar q_1z)F_1(q_1z)\Psi_1^+(q_1z)^{-1}
\Psi_0^+(\bar q_1z)^{-1}
E_1(Cq_1w)\Big|_{w=z}\\
    &=c_1a_{12}\frac{(w-q_1^{-1}z)(w-z)}{(w-q_2z)(w-q_3z)}
F_0(\bar q_1z)F_1(q_1z)   E_1(Cq_1w) \Psi_1^+(q_1z)^{-1}
\Psi_0^+(\bar q_1z)^{-1}
\Big|_{w=z}\\
    &=\frac{c_1a_{12}(1-q_1^{-1})}{(1-q_2)(1-q_3)}F_0(\bar q_1z)\Big((1-\frac{z}{w})F_1(q_1z)E_1(q_1Cw)
\Psi_1^+(q_1z)^{-1}\Big|_{w=z}\Big)\Psi_0^+(\bar q_1z)^{-1}\\
&=F_0(\bar q_1z)\Psi_0^+(\bar q_1z)^{-1}.
\end{align*}
    The computation for $X_{-1}^-(z)$ is similar.
\end{proof}

The currents $X^\pm_i(z)$ are defined using currents of algebra $\E_1$. We have symmetric formulas for $X^\pm_i(z)$ which use currents  of $\check{\E}_1$.

\begin{prop}\label{second way prop} For $r\geq 0$ we have 
    \begin{align*}
     X^+_{-r}(z)&=E_1(z)\,\check{e}^{(r)}(z),\\
     X^+_{r}(z)&=\check{f}^{(r)}(q_1^{r}C^{-1}z)\,\big(\check{\psi}^{+,(r)}(q_1^{r}C^{-1}z)\big)^{-1}E_1(z),\\   
      X^-_{-r}(z)&=(s_2-s_2^{-1})F_1(z)\big(\check{\psi}^{-,(r)}(q_1^{r}C^{-1}z)\big)^{-1}\check{e}^{(r)}(q_1^{r}C^{-1}z), \\
     X^-_{r}(z)&=-(s_2-s_2^{-1})\check{f}^{(r)}(z)\, F_1(z).
\end{align*}
\end{prop}
\begin{proof} We use induction on $r$. 
The case of $r=0$ is trivial.

Consider the first equation.  The case of $r=1$ follows from 
\begin{align*}
(q_1^{-1}
-q_2) 
E_1(q_1^{-1}z) 
E_1(z)=-(1-q_2q_1^{-1})E_1(z)E_1(q_1^{-1}z)
.
\end{align*}

Then using induction hypothesis, commutativity of $e(z)$ with $\check{e}(w)$, and the case $r=1$ one more time, we obtain:
\begin{align*}
E_1(q_1^{-r}z) e^{(r)}(q_1^{-r}z)= E_1(q_1^{-r}z) e^{(r-1)}(q_1^{-r}z)  e(q_1^{-1}z)=E_1(q_1^{-1}z)\check{e}^{(r-1)}(q_1^{-1}z)  
e(q_1^{-1}z)\\
=E_1(q_1^{-1}z)e(q_1^{-1}z) \check{e}^{(r-1)}(q_1^{-1}z) 
=E_1(z)\check{e}(z)\check{e}^{(r-1)}(q_1^{-1}z)
=E_1(z)\check{e}^{(r)}(z).
\end{align*}
The proof of the last equation is similar.

Consider the second equation. Using $f(z)\psi^+(z)^{-1}f(w)\psi^+(w)^{-1}=f(w)f(z)\psi^+(w)^{-1}\psi^+(z)^{-1}$, current $X^+_r(z)$ can be rewritten in the following form:
\begin{align*}
 c_1^{-r}X^+_{r}(Cz)&=f(q_1^{r-1}z)\dots f(q_1z) f(z) \psi^+(q_1^{r-1}z)^{-1}\dots \psi^+(q_1z)^{-1}\psi^{+}(z)^{-1}E_1(q_1^rCz)\\&=f(z)\psi^+(z)^{-1}f(q_1z)\psi^+(q_1z)^{-1}\dots f(q_1^{r-1} z) \psi^+(q_1^{r-1}z)^{-1}E_1(q_1^rCz).
\end{align*}
After that the proof is the same as in the previous case. The case of $r=1$  follows from Lemma \ref{X simpl} and a similar computation for the checked part.
\end{proof}

From the definition and Proposition \ref{second way prop}, we obtain the following recursions for currents $X^\pm_i(z)$.
\begin{cor}\label{inductive X cor}
We have
\begin{align*}
X^+_{i-1}(z)&=c_1X^+_i(q_1^{-1}z)e(q_1^{-1}z)=\check{c}_1X^+_{i}(z)\check{e}(q_1^iz)  &(i\leq 1),\ \ \\
X^-_{i-1}(z)&=c_1X^-_i(q_1 z){\psi^-(C^{-1}z)}^{-1}e(C^{-1}z)
=\check c_1X^-_i(z){\check \psi^-(C^{-1}q_1^{-i+1}z)}^{-1}
\check e(C^{-1}q_1^{-i+1}z)&(i\leq 0),\ \ \\
X^+_{i+1}(z)&=-c_1 f(C^{-1}z){\psi^+(C^{-1}z)}^{-1}X^+_i(q_1 z)=-\check c_1 \check f(C^{-1}q_1^{i+1}z){\psi^+(C^{-1}q_1^{i+1}z)}^{-1}
X^+_i(z) &(i\ge 0),\ \  \\
X^-_{i+1}(z)&=-c_1f(q_1^{-1}z)X^-_i(q_1^{-1} z)=-\check c_1 \check f(q_1^{-i}z)X^-_i(z) &(i\ge -1).
\end{align*} 
\qed 
\end{cor}
Later we show that the equations of Corollary \ref{inductive X cor} are valid for all $i\in\Z$,
see Corollary \ref{cor:rec-all}.
\subsection{Pictorial presentations of fused currents}
It is convenient to picture fused  currents in $\E_2$ as a collection of colored boxes in a plane with coordinates $\bar q_1^{-1}, \bar q_3^{-1}$, see Figure \ref{pic}. Similar pictures were used to  visualize thin $\E_n$ modules in \cite{FJMM1}, \cite{FJMM3}. We expect such pictures will be useful for $\gl_n$ generalizations of the present paper.

\def\wbox at (#1,#2){\draw[thick] (#1,#2)--++(0.5,0)--++(0,0.5)--++(-0.5,0)--(#1,#2);}

\def\bbox at (#1,#2){\filldraw[fill=gray,-] (#1,#2)--++(0.5,0)--++(0,0.5)--++(-0.5,0)--(#1,#2);}

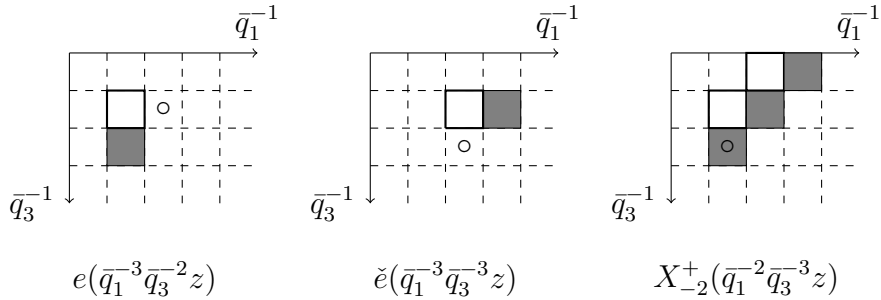
\begin{figure}[ht]
\centering
\begin{tikzpicture}
\draw[ ->] (-6,0) -- (-3.5,0);
\node at (-3.5,0.4) {$\bar q_1^{-1}$};
\node at (-6.5,-2) {$\bar q_3^{-1}$};
\draw[dashed, ] (-6,-0.5) -- (-3.5,-0.5);
\draw[dashed, ] (-6,-1) -- (-3.5,-1);
\draw[dashed, ] (-6,-1.5) -- (-3.5,-1.5);
\draw[ ->] (-6,0) -- (-6,-2);
\draw[dashed, ] (-5.5,0) -- (-5.5,-2);
\draw[dashed, ] (-5,0) -- (-5,-2);
\draw[dashed, ] (-4.5,0) -- (-4.5,-2);
\draw[dashed, ] (-4,0) -- (-4,-2);
\bbox at (-5.5,-1.5);
\wbox at (-5.5,-1);
\node at(-4.75,-0.75) {$\circ$};
\node at (-5, -3) {$e(\bar q_1^{-3}\bar q_3^{-2}z)
$
};

\draw[ ->] (-2,0) -- (0.5,0);
\node at (0.5,0.4) {$\bar q_1^{-1}$};
\node at (-2.5,-2) {$\bar q_3^{-1}$};
\draw[dashed, ] (-2,-0.5) -- (0.5,-0.5);
\draw[dashed, ] (-2,-1) -- (0.5,-1);
\draw[dashed, ] (-2,-1.5) -- (0.5,-1.5);
\draw[ ->] (-2,0) -- (-2,-2);
\draw[dashed, ] (-1.5,0) -- (-1.5,-2);
\draw[dashed, ] (-1,0) -- (-1,-2);
\draw[dashed, ] (-0.5,0) -- (-0.5,-2);
\draw[dashed, ] (0,0) -- (0,-2);
\bbox at (-0.5,-1);
\wbox at (-1,-1);
\node at(-0.75,-1.25) {$\circ$};
\node at (-1, -3) {$\check{e}(\bar q_1^{-3}\bar q_3^{-3}z)$
};

\draw[ ->] (2,0) -- (4.5,0);
\node at (4.5,0.4) {$\bar q_1^{-1}$};
\node at (1.5,-2) {$\bar q_3^{-1}$};
\draw[dashed, ] (2,-0.5) -- (4.5,-0.5);
\draw[dashed, ] (2,-1) -- (4.5,-1);
\draw[dashed, ] (2,-1.5) -- (4.5,-1.5);
\draw[ ->] (2,0) -- (2,-2);
\draw[dashed, ] (2.5,0) -- (2.5,-2);
\draw[dashed, ] (3,0) -- (3,-2);
\draw[dashed, ] (3.5,0) -- (3.5,-2);
\draw[dashed, ] (4,0) -- (4,-2);
\bbox at (2.5,-1.5);
\wbox at (2.5,-1);
\bbox at (3,-1);
\wbox at (3.0,-0.5);
\bbox at (3.5,-0.5);
\node at(2.75,-1.25) {$\circ$};
\node at (3, -3) {$X^+_{-2}(\bar q_1^{-2}\bar q_3^{-3}z)
$};

\end{tikzpicture}
\caption{Fused currents in $\E_2$.}\label{pic}
\end{figure}

In the pictures we represent each operator 
$E_1(\bar q_1^{-i}\bar q_3^{-j}z)$ 
participating in the fused current by  a gray box at the position $(i,j)$.  By convention, the top left box has position $(1,1)$.

Similarly, each current 
$E_0(\bar q_1^{-i}\bar q_3^{-j}z)$ 
is represented by a white box at the position $(i,j)$. The circle shows the shift of the argument $z$ in the name of the fusion current.

Thus, $e(z)$ is a vertical domino, $\check e(z)$
is a horizontal domino, and $X^+_{-r}(z)$, $r\geq 0$, is a staircase picture with $(r+1)$ gray boxes and $r$ white boxes. Cutting along vertical edges we obtain the formulas for $X^+_{-r}$ in terms of $e(z)$ as in the definition.  Cutting along horizontal edges writes  $X^+_{-r}(z)$ in terms of $e(z)$ as in Proposition \ref{second way prop}. See also Corollary \ref{inductive X cor}. Note that in our pictures any two boxes which share an edge have different colors.

We note that such pictures do not give the order in which currents are multiplied to produce a well-defined fused current. The pictures also do not reflect constants in front of operators.

The currents $f(z),\check{f}(z)$ and $X^-_r(z)$, $r\geq 0$, could be represented by the same pictures, where boxes correspond to operators $F_i(z)$.

\medskip 

As an illustration we show how pictures can be used for studying relations. In Figure \ref{E1 pic} we show the commutation relations of $E_1(w)$ with various pictures. For a pictured fused current $A(z)$, we
put bullets for the position of zeroes and crosses for the positions of poles in the $z$ plane in the informal relation $A(z)E_1(w)=f(z)E_1(w)A(z)$. For example, we have a relation $(z'-\bar q_2w)E_1(z')E_1(w)=(w-\bar q_2z')E_1(w)E_1(z')$ which is depicted in the left most picture as cross in the position $w=q_2^{-1}z'$ and zero at $w=\bar q_2z'$. Here the circle indicates, $z'=\bar q_1^{-2}\bar q_3^{-3}z$. The second picture corresponds to the relation $(z'-\bar q_1w)(z'-\bar q_3w) E_0(z')E_1(w)+(w-\bar q_1z')(w-\bar q_3z')E_1(w)E_0(z')=0$ and the rightmost picture describes the relation of $E_1(w)$ with $\check{e}(z')$ which has the form 
$$(z'-w)(z'-\bar q_1\bar q_3^{-1}w)\check{e}(z') E_1(w)+ (z'-\bar q_1^2w)(z'-\bar q_2w)E_1(w) \check{e}(z')=0.$$

The circle is still showing the shift $c$ in $A(cz)$.

Note that a zero and a cross in the same position can be canceled. (Some justification is required, see Remark \ref{expand remark}, \cite{FJMM3}, which we ignore here.)

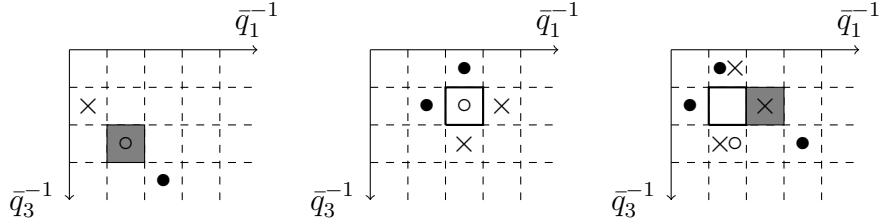
\begin{figure}[ht]
\centering
\begin{tikzpicture}
\draw[ ->] (-6,0) -- (-3.5,0);
\node at (-3.5,0.4) {$\bar q_1^{-1}$};
\node at (-6.5,-2) {$\bar q_3^{-1}$};
\draw[dashed, ] (-6,-0.5) -- (-3.5,-0.5);
\draw[dashed, ] (-6,-1) -- (-3.5,-1);
\draw[dashed, ] (-6,-1.5) -- (-3.5,-1.5);
\draw[ ->] (-6,0) -- (-6,-2);
\draw[dashed, ] (-5.5,0) -- (-5.5,-2);
\draw[dashed, ] (-5,0) -- (-5,-2);
\draw[dashed, ] (-4.5,0) -- (-4.5,-2);
\draw[dashed, ] (-4,0) -- (-4,-2);
\bbox at (-5.5,-1.5);
\node at(-5.25,-1.25) {$\circ$};
\node at(-5.75,-0.75) {$\times$};
\node at(-4.75,-1.75) {$\bullet$};

\draw[ ->] (-2,0) -- (0.5,0);
\node at (0.5,0.4) {$\bar q_1^{-1}$};
\node at (-2.5,-2) {$\bar q_3^{-1}$};
\draw[dashed, ] (-2,-0.5) -- (0.5,-0.5);
\draw[dashed, ] (-2,-1) -- (0.5,-1);
\draw[dashed, ] (-2,-1.5) -- (0.5,-1.5);
\draw[ ->] (-2,0) -- (-2,-2);
\draw[dashed, ] (-1.5,0) -- (-1.5,-2);
\draw[dashed, ] (-1,0) -- (-1,-2);
\draw[dashed, ] (-0.5,0) -- (-0.5,-2);
\draw[dashed, ] (0,0) -- (0,-2);
\wbox at (-1,-1);
\node at(-0.75,-1.25) {$\times$};
\node at(-0.25,-0.75) {$\times$};
\node at(-0.75,-0.75) {$\circ$};
\node at(-0.75,-0.25) {$\bullet$};
\node at(-1.25,-0.75) {$\bullet$};

\draw[ ->] (2,0) -- (4.5,0);
\node at (4.5,0.4) {$\bar q_1^{-1}$};
\node at (1.5,-2) {$\bar q_3^{-1}$};
\draw[dashed, ] (2,-0.5) -- (4.5,-0.5);
\draw[dashed, ] (2,-1) -- (4.5,-1);
\draw[dashed, ] (2,-1.5) -- (4.5,-1.5);
\draw[ ->] (2,0) -- (2,-2);
\draw[dashed, ] (2.5,0) -- (2.5,-2);
\draw[dashed, ] (3,0) -- (3,-2);
\draw[dashed, ] (3.5,0) -- (3.5,-2);
\draw[dashed, ] (4,0) -- (4,-2);
\wbox at (2.5,-1);
\bbox at (3,-1);
\node at(2.85,-1.25) {$\circ$};
\node at(2.65,-1.25) {$\times$};
\node at(2.85,-0.25) {$\times$};
\node at(3.25,-0.75) {$\times$};
\node at(2.25,-0.75) {$\bullet$};
\node at(3.75,-1.25) {$\bullet$};
\node at(2.65,-0.25) {$\bullet$};
\end{tikzpicture}
\caption{Commutation with $E_1$.}\label{E1 pic}
\end{figure}

Similarly, in Figure \ref{E0 pic} we show the commutation relations of $E_0(w)$ with various pictures. 

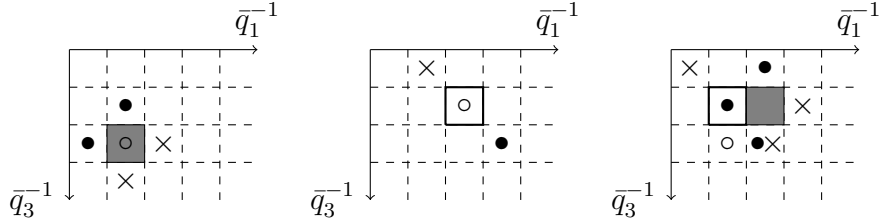
\begin{figure}[ht]
\centering
\begin{tikzpicture}
\draw[ ->] (-6,0) -- (-3.5,0);
\node at (-3.5,0.4) {$\bar q_1^{-1}$};
\node at (-6.5,-2) {$\bar q_3^{-1}$};
\draw[dashed, ] (-6,-0.5) -- (-3.5,-0.5);
\draw[dashed, ] (-6,-1) -- (-3.5,-1);
\draw[dashed, ] (-6,-1.5) -- (-3.5,-1.5);
\draw[ ->] (-6,0) -- (-6,-2);
\draw[dashed, ] (-5.5,0) -- (-5.5,-2);
\draw[dashed, ] (-5,0) -- (-5,-2);
\draw[dashed, ] (-4.5,0) -- (-4.5,-2);
\draw[dashed, ] (-4,0) -- (-4,-2);
\bbox at (-5.5,-1.5);
\node at(-5.25,-1.25) {$\circ$};
\node at(-5.25,-0.75) {$\bullet$};
\node at(-4.75,-1.25) {$\times$};
\node at(-5.75,-1.25) {$\bullet$};
\node at(-5.25,-1.75) {$\times$};

\draw[ ->] (-2,0) -- (0.5,0);
\node at (0.5,0.4) {$\bar q_1^{-1}$};
\node at (-2.5,-2) {$\bar q_3^{-1}$};
\draw[dashed, ] (-2,-0.5) -- (0.5,-0.5);
\draw[dashed, ] (-2,-1) -- (0.5,-1);
\draw[dashed, ] (-2,-1.5) -- (0.5,-1.5);
\draw[ ->] (-2,0) -- (-2,-2);
\draw[dashed, ] (-1.5,0) -- (-1.5,-2);
\draw[dashed, ] (-1,0) -- (-1,-2);
\draw[dashed, ] (-0.5,0) -- (-0.5,-2);
\draw[dashed, ] (0,0) -- (0,-2);
\wbox at (-1,-1);
\node at(-0.25,-1.25) {$\bullet$};
\node at(-0.75,-0.75) {$\circ$};
\node at(-1.25,-0.25) {$\times$};

\draw[ ->] (2,0) -- (4.5,0);
\node at (4.5,0.4) {$\bar q_1^{-1}$};
\node at (1.5,-2) {$\bar q_3^{-1}$};
\draw[dashed, ] (2,-0.5) -- (4.5,-0.5);
\draw[dashed, ] (2,-1) -- (4.5,-1);
\draw[dashed, ] (2,-1.5) -- (4.5,-1.5);
\draw[ ->] (2,0) -- (2,-2);
\draw[dashed, ] (2.5,0) -- (2.5,-2);
\draw[dashed, ] (3,0) -- (3,-2);
\draw[dashed, ] (3.5,0) -- (3.5,-2);
\draw[dashed, ] (4,0) -- (4,-2);

\wbox at (2.5,-1);
\bbox at (3,-1);

\node at(2.75,-1.25) {$\circ$};
\node at(3.15,-1.25) {$\bullet$};
\node at(3.25,-0.25) {$\bullet$};

\node at(2.25,-0.25) {$\times$};
\node at(2.75,-0.75) {$\bullet$};
\node at(3.75,-0.75) {$\times$};
\node at(3.35,-1.25) {$\times$};

\end{tikzpicture}
\caption{Commutation with $E_0$.}\label{E0 pic}
\end{figure}

Now using the rightmost pictures in Figures \ref{E1 pic} and \ref{E0 pic}, we readily see that current $E_1(w)E_0(q_3w)$ (vertical domino corresponding to $e(w)$) commutes with horizontal domino (corresponding to $\check{e}(z)$) as all crosses and bullets cancel.

\section{Algebra $\mc A_0$}\label{A sec}
We give another realization of quantum toroidal $\E_2$.
\subsection{Definition}

We choose $\beta,\beta^\vee$ such that 
\begin{align}\label{beta}
     q_3=q_1^{-\beta} ,\qquad \check{q}_3=\check{q}_1^{-\check{\beta}}, \qquad \beta+\beta^\vee=2.
\end{align}

\medskip

Denote $\E_1=\E_1(q_1,q_2,q_3)$ and  $\check{\E}_1=\E_1(\check{q}_1,\check{q}_2,\check{q}_3)$ the quantum toroidal $\gl_1$ algebras with parameters $q_i$ and $\check{q}_i$ respectively. 
Denote $\E_1[K]=\E_1\otimes\C[K^{\pm1}]$  
and 
$\check{\E}_1[\check{K}]=\check \E_1\otimes\C[\check{K}^{\pm1}] 
$ 
the extensions of $\E_1$ and $\check{\E}_1$ by split central group-like invertible elements $K$ and $\check{K}$.

Let
\begin{align*}
\mc K_0=\E_1[K]\otimes\check{\E}_1[\check{K}]/(\psi_0\otimes 1-1\otimes \check{\psi_0}, C\otimes 1-1\otimes \check{C}).
\end{align*}
The algebra $\mc K_0$ is the tensor product of two (commuting) quantum toroidal $\gl_1$ with identified central elements extended by invertible split   invertible group-like central elements $K, \check{K}$.  
The algebra ${\mc K}_0$ plays a role of an affinization of Cartan subalgebra of $U_{s_2}\widehat{\gl}_2$.

\medskip

 Recall function $\omega(x)$, see \eqref{om}. Set
\begin{align*} 
\check{\omega}(x)=\frac{(1-\check{q}_2x)(1-\check{q}_3x)}{(1-x)(1-\check{q}_2\check{q}_3 x)}\,.
\end{align*}

\medskip 

Algebra $\calA_0=\mc A_0(q_1,q_2,q_3)$ is a unital associative subalgebra of $\End\mc{M}$ which is an extension of the algebra $\mc K_0$
by elements $X^{\pm}_{i,k}$, $i,k\in\Z$ satisfying three groups of relations. To write the relations, we collect the extra generators into the generating series
\begin{align}\label{X series}
&X_i^\pm(z)=\sum_{k\in\Z}X_{i,k}^\pm z^{-k}.
\end{align}
The first group of relations is between $X_i^\pm(z)$ and $\mc K_0$. (We assign names to the relations which are given in parenthesis.) 
\begin{align*}
&\psi_0,C \text{  are central,}\\
(KX^+_i):\quad &KX^+_i(z)K^{-1}=q_2^{\beta/2}X^+_i(z)\,,\quad 
 KX^-_i(z)K^{-1}=q_2^{-\beta/2}X^-_i(z)\,,
\\
(\check{K}X^+_i):\quad &\check{K}X^+_i(z)\check{K}^{-1}=\cq_2^{\cb/2}X^+_i(z)\,,\quad 
\check{K}X^-_i(z)\check{K}^{-1}=\cq_2^{-\cb/2}X^-_i(z)\,,
\end{align*}
\begin{align*}
(eX^+_i):\quad &e(u)X^+_i(z)-\omega^-\Bigl(\frac{z}{u}\Bigr)X^+_i(z)e(u)=
\frac{1}{ s_1^{-1}- s_1}\delta\Bigl(\frac{z}{u}\Bigr)X^+_{i-1}(q_1z)\,,
\\ 
(fX^+_i):\quad &[f(u),X^+_i(z)]=\frac{1}{s_1^{-1}- s_1}
\delta\Bigl(\frac{q_1^{-1}z}{C u}\Bigr)X^+_{i+1}(q^{-1}_1z)\psi^+(u)\,,
\\
(\psi X^+_i):\quad &\psi^\pm(u)X^+_i(z){\psi^\pm(u)}^{-1}=\omega^\pm\Bigl(\frac{z}{C^{(1\pm 1)/2} u}\Bigr)X^+_i(z)\,,
\end{align*}
\begin{align*}
(\check{e}X^+_i):\quad &\check{e}(u)X^+_i(z)-\check{\omega}^-\Bigl(\frac{ q_1^{i}z}{u}\Bigr)
X^+_i(z)\check{e}(u)=
\frac{1}{\check{s}_1^{-1}-s_1}
\delta\Bigl(\frac{q_1^{i}z}{u}\Bigr)X^+_{i-1}(z)\,,
\\ 
(\check{f}X^+_i):\quad &[\check{f}(u),X^+_i(z)]=\frac{1}{\check{s}_1^{-1}-\check{s}_1}
\delta\Bigl(\frac{q_1^{i+1}z}{Cu}\Bigr)
X^+_{i+1}(z)\check{\psi}^+(u)\,,
\\
(\check{\psi}X^+_i):\quad&\check{\psi}^\pm(u)X^+_i(z){\check{\psi}^\pm(u)}^{-1}
=\check{\omega}^\pm\Bigl(\frac{q_1^{i} z}{C^{(1\pm1)/2}u}\Bigr)X^+_i(z)\,,
\end{align*}
\begin{align*}
(eX^-_i):\quad &[X^-_i(z),e(u)]=\frac{1}{s_1^{-1}-s_1}\delta\Bigl(\frac{z}{C q_1u}\Bigr) 
\psi^-(u) X^-_{i-1}(q_1^{-1}z)\,,
\\ 
(fX^-_i):\quad &X^-_i(z)f(u)-\omega^+\Bigl(\frac{z}{u}\Bigr)f(u)X^-_i(z)=\frac{1}{ s_1^{-1}- s_1}
\delta\Bigl(\frac{z}{u}\Bigr)X^-_{i+1}(q_1z)\,,
\\
(\psi X^-_i):\quad &{\psi^\pm(u)}^{-1}X^-_i(z)\psi^\pm(u)=\omega^\pm\Bigl(\frac{z}{C^{(1\mp 1)/2}u}\Bigr)X^-_i(z)\,,
\end{align*}
\begin{align*}
(\check{e}X^-_i):\quad &[X^-_i(z),\check{e}(u)]
=\frac{1}{\check{s}_1^{-1}-\check{s}_1 }
\delta\Bigl(\frac{q_1^{1-i}z}{Cu}\Bigr)
\check{\psi}^{-}(z)X^-_{i-1}(z)\,,
\\ 
(\check{f}X^-_i):\quad &X^-_i(z)\check{f}(u)-\check{\omega}^+\Bigl(\frac{q_1^{-i}z}{u}\Bigr)
\check{f}(u)X^-_i(z)=
\frac{1}{\check{s}_1^{-1}-\check{ s}_1}
\delta\Bigl(\frac{ q_1^{-i}z}{u}\Bigr)X^-_{i+1}(z)
\,,
\\
(\check{\psi} X^-_i):\quad &{\check{\psi}^\pm(u)}^{-1}X^-_i(z)\check{\psi}^\pm(u)
=\check{\omega}^\pm\Bigl(\frac{q_1^{-i}z}{C^{(1\mp1)/2}u}\Bigr)X^-_i(z)\,.
\end{align*}

\medskip

The second group of relations is between $X_i^{\pm}(z)$ and $X_j^{\pm}(w)$.

\noindent 
\begin{align*}
&(z-q_2w) X^+_i(z)X^+_i(w)+
(w-q_2z) X^+_i(w)X^+_i(z)
=0\,,\\
&(q_2z-w) X^-_i(z)X^-_i(w)+
(q_2w-z) X^-_i(w)X^-_i(z)
=0\,,
\end{align*}
\begin{align*}
\prod_{r=1}^{j-i-1}(z-q_1^{r}w)
X^+_i(z)X^+_j(w)+(-1)^{j-i}q_2 
\,
p^+_{j-i}(z/w)
\prod_{r=0}^{j-i-2}(q_3^{-1}q_1^{r}w-z)
 X^+_j(w)X^+_i(z)=0,\\
\prod_{r=1}^{j-i-1}(w-q_1^{r}z)
X^-_i(z)X^-_j(w)+(-1)^{j-i}q_2 
\,
p^-_{j-i}(w/z)
\prod_{r=0}^{j-i-2}(q_3^{-1}q_1^{r}z-w)
 X^-_j(w)X^-_i(z)=0,\\
\end{align*}
where $j>i$ and
\begin{align*}
p_s(u)=
\prod_{r=0}^{s}\frac{q_2^{-1}q_1^{r}-u}{q_1^r-u}.
\end{align*}

The third and last group of relations is
between $X_i^{\pm}(z)$ and $X_j^{\mp}(w)$.

\begin{align}
&[X^+_i(z),X^-_j(w)]=
\label{EF}
\\
&
\sum_{r+\check{r}=i+j\atop r,\check{r}\in\Z_{\ge0}} \psi_0^{-i}
\delta\Bigl(\frac{C w}{q_1^rz}\Bigr) 
\, k_r^+(q_1^{-r}w)
\check{k}_{\check{r}}^+(\check{ q}_1^{-\check{r}} q_1^{-j}w)   -\sum_{r+\check{r}=-i-j\atop r,\check{r}\in\Z_{\ge0}} \psi_0^{-j}
\delta\Bigl(\frac{C z}{q_1^rw}\Bigr)
\,k^-_r( q_1^{-r}z) 
\check{k}_{\check{r}}^-(\check{q}_1^{-\check{r}} q_1^iz)    
\,,\nn
\end{align}
where $k_r^\pm(z), \check k_r^\pm(z)$ are given by \eqref{k} and the same formulas with checks.

\medskip

\medskip 
Algebra $\mc A_0$  has a homogeneous $\Z$-grading which we call degree and denote by $\deg$. The degrees of elements in $\mc K_0$ are defined  by\eqref{hom grading} and same formula for $\check{\E}_1$. In addition, we set
\begin{align}\label{hom grading 3}
\deg 
X_{i,k}^\pm=k.
\end{align}

Algebra $\mc A_0$  has another $\Z$-grading which we call weight and denote by $\wt$. The weights of elements in $\mc K_0$ are zero. In addition, we set
\begin{align}\label{weight 3}
\wt
X_{i,k}^\pm=\pm 1.
\end{align}

An $\mc A_0$ module $W$ is 
said to be graded if 
$W=\oplus_{d_1,d_2\in\Z} W_{d_1,d_2}$,  $\dim W_{d_1,d_2}<\infty$, and for any $x\in \mc A_0$ of degree $a$ and weight $b$ we have  $x W_{d_1,d_2}\subset W_{d_1+a,d_2+b}$.

An $\mc A_0$ module $W$ is 
said to be admissible if it is graded and  for each $d_2\in \Z$ there exists $d$ (depending on $d_2$) such that such that $W_{d_1,d_2}=0$ for all $d_1>d$. 

Note that an admissible $\mc A_0$ module is a direct sum of admissible $\mc K_0$ modules.

\begin{rem}\label{waiver remark}{\rm
We note one more time that, similarly to vertex operator algebras,  $\mc A_0$ is not an algebra in the usual sense as relations \eqref{EF} involve infinite sums and make sense only as operators on an admissible module. Therefore, we effectively study only sets of operators satisfying the relations of $\mc A_0$ and acting on admissible modules. The same applies to algebras $\mc A_N$ in Section \ref{shifted sec}. \qed}
\end{rem} 

\begin{rem}\label{expand remark}{\rm 
We make a remark which seems to be ignored in some texts. 

Let $a(z)=\sum_{k\in\Z} a_kz^{-k}$ and  $b(z)=\sum_{k\in\Z} b_k z^{-k}$ be two currents acting on a graded vector space $\mc M$. Assume that the action is graded, $\deg a_k=\deg b_k=k$ for all $k$, and ``admissible", for any vector $w\in \mc M$, $a_kw=b_kw=0$ for large enough $k$. Consider a relation of the form $$g_1(z,w)a(z)b(w)=g_2(w,z)b(w)a(z),$$
where $g_i(z,w)$ are some homogeneous rational functions.

Suppose that $a(z)$ is a one sided current: $a_k=0$ for $k\gg 0$. Then  all matrix coefficients of the current $a(z)b(w)$ are Laurent polynomials. In particular,
if $g_1,g_2$ are homogeneous polynomials of $z,w$ and  $g(u)$ is a non-zero rational function, then the relation above is equivalent to the relation $g^+(z/w)g_1(z,w)a(z)b(w)=g^+(z/w)g_2(w,z)b(w)a(z).$

However, if both currents $a(z)$, $b(z)$ are two sided then a relation cannot be multiplied or divided without changing the meanings. For example, 
consider the relation 
$$
a(z) b(w)=\frac{w-q_2z}{z-q_2w}\, b(w) a(z).
$$
Then to make sense we need to expand the rational function in the right hand side as a series. To avoid infinite summation this expansion must be in powers of $z/w$, that is in the region $|z|\ll|w|$. Now, this relation implies the matrix coefficients of $a(z)b(w)$ are Laurent polynomials and, in particular, $b(w)a(q_2w)=0$. In contrast, the relation 
$$
(z-q_2w)a(z) b(w)=(w-q_2z) b(w) a(z)
$$
does not imply that $b(w)a(q_2w)=0$ and in fact  $b(w)a(q_2w)$ can be non-zero while $a(z)b(w)$ can have a pole at $z=q_2w$. To divide by $(z-q_2w)$, one must include the term $\delta(q_2w/z)$ multiplied by the residue 
$\res_{z=q_2w}a(z)b(w)dz/z$.

For example, if $\mc M$ is an admissible $\mc A_0$ module, then relation $(eX^+_i)$ says the following.
\begin{enumerate}
\item As rational functions
\begin{align*} 
&e(u)X^+_i(z)=\omega^-\Bigl(\frac{z}{u}\Bigr) X^+_i(z)e(u)\,.
\end{align*}
That means that all matrix coefficients of either side can be summed up to rational functions and the two rational functions coincide. 

\item The only pole of either side in $u\in \C^\times$ is
a simple pole $u=z$. Note that as a series the two sides are expanded in opposite direction resulting the delta function term in relation $(eX^+_i)$.

\item The residue of either side is
\begin{align*}
\Res_{u=z}e(u)X^+_i(z)\frac{du}{u}=\Res_{u=z} \omega^-\Bigl(\frac{z}{u}\Bigr) X^+_i(z)e(u) \frac{du}{u}= \frac{1}{s_1^{-1}-s_1}X^+_{i-1}(q_1z)\,. 
\end{align*}
\end{enumerate}

Finally, we note that everywhere in this paper in the relations between two-sided currents, we expand $g_1(z,w)$ in powers of $w/z$ and $g_2(w,z)$ in powers of $z/w$. However, we decided to carry pluses and minuses which indicate the direction of expansion explicitly to remember about that issue.} \qed
\end{rem}

The first group of relations may look cumbersome, but it has the following natural interpretation. 

The left and right adjoint actions ${\rm ad}$ and  ${\rm ad}^R$ of algebra $\E_1$  on $\mathcal A_0$ are given by 
\begin{align*}
    {\rm ad}(x) A= \sum x'AS_1(x''), \qquad {\rm ad}^R(x) A= \sum S_1(x')Ax'',
\end{align*}
where $A\in{\mathcal A}_0$, $x\in\E_1$, and $\Delta_1(x)=\sum x'\otimes x''$. 

Recall the left and right vector representations $V_1(z)$, $V_1^R(z)$ of $\E_1$, see \eqref{vector rep}. 

Then relations $(\Psi X^+_i)$, $(eX^+_i)$,$(fX^+_i)$ mean that the map $V_1(z)\to \mathcal A_0$ mapping $\ket{ i,z }
\mapsto X_{-i}^+(q_1^iz)$ 
is a map of left $\E_1$ modules.
Similarly, relations $(\Psi X^-_i)$, $(eX^-_i)$,$(fX^-_i)$ mean that the map $V_1^R(z)\to \mathcal A_0$ mapping $\ket{ i,z } \mapsto X_{i}^-(q_1^iz)$ 
is a map of right $\E_1$ modules.

Algebra $\mathcal A_0$  is also a left and right $\check{\E}_1$ module under left and right adjoint actions. 
Then relations $(\check{\Psi} X^+_i)$, $(\check{e}X^+_i)$,$(\check{f}X^+_i)$ mean that the map $\check{V}_1(z)\to \mathcal A_0$ mapping 
$\ket{ i,z }^\vee 
\mapsto X_{-i}^+(z)$ is a map of left $\check{\E}_1$ modules.
Similarly, relations $(\check{\Psi} X^-_i)$, $(\check{e}X^-_i)$,$(\check{f}X^-_i)$ mean that the map $\check{V}^R_1(z)\to \mathcal A_0$ mapping 
$\ket{ i,z }^\vee \mapsto X_{i}^-(z)$   
is a map of right $\check{\E}_1$ modules.

\begin{rem} {\rm In quantum affine algebra $U_{s_2}(\widehat{\mathfrak{gl}}_2)$, there are relations of the form $[H_{1,-1},E_i]=b_iE_{i-1}$ and $[H_{1,1},E_i]=b_i'E_{i+1}$, where $b_i, b_i'$ are constants depending on conventions and central element. We view relations $(eX^+_i)$ and $(fX^+_i)$ as affine analogs of these relations. 

Then the relations $(eX^-_i)$ and $(fX^-_i)$ correspond to commutators $[H_{1,\mp1},F_i]$ which are proportional to $F_{i\mp 1}$. 
The relations $(\check{e}X^\pm_i)$ and $(\check{f}X^\pm_i)$ correspond to commutators with the second set of Cartan generators $H_{0,\pm1}$.

We note that, similarly to Drinfeld new realization of $U_{s_2}(\widehat{\mathfrak{gl}}_2)$, the algebra $\mc A_0$ has no cubic Serre relations.
} \qed
\end{rem}

\subsection{The isomorphisms}
We show that image of $\mc A_0$ is isomorphic to the image of $\E_2$ as graded subalgebras of endomorphisms of an admissible module. We recall that everywhere in this paper parameters $q_i$ of $\mc A_0$ and parameters $\bar q_i$ of $\E_2$ are related via \eqref{q-bar q}.

To continue, we need fractional powers of $\Psi_{0,0}$. So, we extend $\E_2$ with an invertible group-like element $K$. We demand that element $K$ has the same commutation relations with $\E_2$ as $\Psi_{0,0}^{\beta/4}$, where $\beta$ is defined in \eqref{beta} 
We denote the resulting algebra by $\E_2[K]$. 

We warn the reader that $K$ is not central in contrast to $\E_1[K]$. We have an automorphism of $\E_2[K]$ which is identity on $\E_2$ and rescales $K$ by a non-zero scalar.
The action of $K$ can be defined on any admissible $\E_2$ module. For example,  one can set $K=\bar K\Psi_{0,0}^{\beta/4}$ where $\bar K$ is any non-zero constant.

Note that $\Psi_{0,0},\Psi_{1,0}$, and $ K$ commute with subalgebras $\E_1$ and $\check{\E}_1$ of $\E_2$.

Let $\mathcal{M}$ be an admissible $\mc A_0$ module. 
Define a map $\iota:\  \E_2[K] \to \mathcal{A}_0$ given on generators by 
\begin{align*}
\Psi^\pm_1(z) &\mapsto  k_0^\pm(z)\check{k}_0^\pm(z), \qquad\qquad\, \Psi^\pm_0(z) \mapsto \psi_0^{\pm1}\big(k_0^\pm(\bar q_1^{-1}z)\check{k}_0^\pm(\bar q_3^{-1}z)\big)^{-1}, \\
E_1(z) &\mapsto X^+_0(z), \qquad  \qquad \qquad\ \ E_0(z) \mapsto \frac{1}{s_2-s_2^{-1}}\psi_0^{-1}\big(k_0^-(\bar q_1^{-1}z)\check{k}_0^-(\bar q_3^{-1}z)\big)^{-1}X^-_{-1}(\bar q_1^{\,-1}Cz) , \\
F_1(z) &\mapsto \frac{1}{s_2-s_2^{-1}}X^-_0(z),\qquad \quad  
F_0(z) \mapsto 
\psi_0 X^+_1(\bar q_1^{\,-1}Cz) \big(k_0^+(q_1^{-1}z)\check{k}_0^+(\bar q_3^{-1}z)\big)^{-1},
\end{align*} 
and $K\mapsto K$. 

We extend $\iota$ to the fused currents in $\E_2$.
\begin{thm}\label{iota thm}
    The map $\iota$ is a well-defined homomorphism of graded algebras. Moreover, it maps
    \begin{align*}
e(z) \mapsto e(z), \quad \check{e}(z)\mapsto \check{e}(z), \quad  f(z) \mapsto f(z), \quad \check{f}(z)\mapsto \check{f}(z).
\end{align*}
The extended map $\iota$ is surjective.
\end{thm}

We prove Theorem \ref{iota thm} in Appendix, see Section \ref{app iota sec}.

Theorem \ref{iota thm} says that if $\mc M$ is an admissible $\mc A_0$ module then it has a structure of admissible $\E_2$ module such that the images of $\mc E_2$ extended by fusion currents and of $\mc A_0$ in $\End \mc M$ coincide.

\medskip

We also have the inverse map mapping currents to the currents with the same names. More precisely, let $\mathcal{M}_2$ be an admissible $\E_2$ module. Let $\tilde \E_2^{\mc M_2}[K]$ be the image of $\E_2[K]$ in  $\End\mathcal{M}_2$ extended by fusion currents.
Define a map $\tau:\ \mc A_0 \to \tilde \E_2^{\mc M_2}[K]$ given on generators by 

\begin{align*}
e(z) \mapsto e(z), \quad \check{e}(z)\mapsto \check{e}(z), \quad  f(z) \mapsto f(z), \quad \check{f}(z)\mapsto \check{f}(z),\\
K\mapsto K, \quad \check{K}\mapsto K^{-1}\Psi_{0,0},\quad C\to C, \quad X^\pm_i(z) \mapsto X^\pm_i(z).
\end{align*}

\begin{thm}\label{tau thm}
    The map $\tau$ defines a structure of admissible $\mc A_0$ module on $\mc M_2$. Moreover, the map $\tau$ is graded and the image of $\tau$ contains $\E_2$.
\end{thm}

We prove Theorem \ref{tau thm} in Appendix, see Section \ref{app tau sec}.

Theorem \ref{tau thm} says that if $\mc M_2$ is an admissible $\mc E_2$ module then it has a structure of admissible $\mc A_2$ module such that images of $\mc E_2$ extended by fusion currents and of $\mc A_0$ in $\End \mc M_2$ coincide.

\section{Algebra $\mc A_N$} \label{shifted sec}
We define and study the shifted version of algebra $\mc A_0$.  We define it only for the case $\psi_0=1$ as it seems that there is no natural way to keep non-trivial $\psi_0$ to preserve various properties, such as Theorem \ref{coproduct thm} below. This is in parallel to shifted quantum toroidal algebras of \cite{N} which are defined at level one.
\subsection{Definition}

Let $N$ be an
integer. We think of $N$ as 
an integral $\slt$ weight.

Let first $N$ be an even integer. 
Let $\mc A_N=\mc A_N(q_1,q_2,q_3)$ be an algebra generated with the same generators and relations as $\mc A_0$ (where $\psi_0=1$) with two changes: algebra $\mc K_0$ is replaced with  algebra $\mc K_N$ given by
\begin{align}\label{K change}
\mc K_N=\E_1[K]\otimes\check{\E}_1[\check{K}]/(\psi_0\otimes 1-1,1\otimes \check{\psi_0}-1, C\otimes 1-q_1^{N/2} 1\otimes \check{C})
\end{align}
and  relations \eqref{EF} are replaced by
\begin{align}
&[X^+_i(z),X^-_j(w)]=
\label{EFshifted}
\\
&
\sum_{r+\check{r}=i+j+N/2\atop r,\check{r}\in\Z_{\ge0}}
\delta\Bigl(\frac{C w}{q_1^rz}\Bigr) 
\, k_r^+(q_1^{-r}w)
\check{k}_{\check{r}}^+(\check{ q}_1^{-\check{r}} q_1^{-j}w) - (-1)^N\sum_{r+\check{r}=-i-j+N/2\atop r,\check{r}\in\Z_{\ge0}}
\delta\Bigl(\frac{C z}{q_1^rw}\Bigr)
\,k^-_r( q_1^{-r}z) 
\check{k}_{\check{r}}^-(\check{q}_1^{-\check{r}} q_1^iz)  
\,.\nn
\end{align}

In other words, we changed the identification of central elements $C, \check{C}$ of $\E_1,\check{\E}_1$ to
\begin{align*}
C=q_1^{N/2}\check C,
\end{align*}
and ``shifted" the relations between $X^+_i(z)$ and $X^-_j(z)$.

For an odd $N$ in addition to \eqref{K change} and \eqref{EFshifted}, we also change the labeling of $X_{i,k}^\pm$. 
Namely, we use generators $X^+_{i+1/2,k}$ and $X_{i,k}^-$, $i,k\in\Z$. Thus, the first index of $X^+$ is a half integer and the first index of $X^-$ is an integer.

\medskip

Note that if $N<0$ then 
$$
[X^+_i(z),X^-_j(w)]=0 \qquad \text{if} \qquad \frac{N}{2}<i+j<-\frac{N}{2}.
$$
We call the case of $N<0$ antidominant and the case of $N>0$  dominant.

The algebra $\mc A_N$ has degree $\deg$ and weight $\wt$ defined in the same way as in $\mc A_0$, see \eqref{hom grading 3} and \eqref{weight 3}. All maps below respect both degree and weight.

The algebra $\mc A_N$ seems to be different from shifted quantum toroidal algebra of \cite{N}.

\medskip 

We finish this subsection by giving a slightly different form of algebras $\mc A_N$ obtained by essentially relabeling of generators.

For even $N$, let $\mc A_N^-$ be the algebra generated by $\mc K_N$ and $X^\pm_{i+1/2}(z)$, where $i\in\Z$, with  the same relations as $\mc A_N$. 

For odd $N$, let $\mc A_N^-$ be the algebra generated by $\mc K_N$ and $X^+_i(z)$, $X^-_{i+1/2}(z)$, where $i\in\Z$, with  the same relations as $\mc A_N$. 

\begin{lem}\label{shift lem}
There is an algebra isomorphism $\mc A_N^{-}\to \mc A_N$ sending
\begin{align*}
X^+_{i}(z)&\mapsto X^+_{i+ 1/2}(z), \qquad  X^-_{i}(z)\mapsto X^-_{i-1/2}(z), \\
e(z)&\mapsto e(z), \hspace{63pt} f(z)\mapsto f(z), \\
\check{e}(z)&\mapsto \check{e}(s_1z), \hspace{54pt} \check{f}(z)\mapsto \check{f}(s_1z).
\end{align*}
\end{lem}
\begin{proof} Note that the maps in the lemma restricted to $\E_1$ are identity, while restricted to $\check{\E}_1$ they are the shift of spectral parameter isomorphisms \eqref{shift iso}. In particular, ${k}_r^\pm(z)$ are mapped to ${k}_r^\pm(z)$ and 
$\check{k}_r^\pm(z)$  to $\check{k}_r^\pm(s_1z)$. The checks of relations are straightforward. Clearly, the maps are surjective and invertible. 
\end{proof}

\begin{cor}\label{aut cor}
We have an automorphism $\mc A_N\to\mc A_N$ sending 
\begin{align*}
X^+_{i}(z)&\mapsto X^+_{i+ 1}(z), \qquad  X^-_{i}(z)\mapsto X^-_{i-1}(z), \\
e(z)&\mapsto e(z), \hspace{54pt}  f(z)\mapsto f(z), \\
\check{e}(z)&\mapsto \check{e}(q_1z), \hspace{45pt}  \check{f}(z)\mapsto \check{f}(q_1z).
\end{align*}
\end{cor} 
\begin{proof}
One can either check the corollary directly or apply Lemma \ref{shift lem} twice.
\end{proof}

\begin{rem}{\rm
The automorphism in Corollary \ref{aut cor} is an affine analog of the automorphism of 
$U_q(\widehat{\mathfrak{gl}}_2)$, sending $E(z)\mapsto z E(z)$ and $F(z)\mapsto z^{-1} F(z)$.}
\qed
\end{rem}
\medskip 

We have another useful isomorphism which inverts $\bar q_1$ and exchanges $\E_1$ and $\check{\E}_1$.

\begin{lem}\label{switch iso lem}
There exists an isomorphism of algebras  $\mc A_N( q_1, q_2, q_3)\to\mc A_N(\check{q}_1, \check{q}_2, \check{q}_3)$ mapping 
\begin{align*}
&X_i^\pm(z)\to X_i^\pm(q_1^{\pm i}z), \\ e(z)\mapsto \check{e}(z), \quad
f(z)&\mapsto \check{f}(z),\quad
 \check{e}(z)\mapsto e(z),\quad
\check{f}(z) \mapsto f(z),
\end{align*}
and $C\mapsto \check{C}$, $K\mapsto \check{K}$, $\check{K}\mapsto K$. 
\end{lem}
\begin{proof}
All the checks are straightforward.
\end{proof}

\subsection{Representations of $\mc A_N$} 
We prove a theorem which allows us to construct a large family of admissible $\mathcal A_N$ modules for a dominant shift $N>0$. We will continue the study of $\mathcal A_N$ modules in \cite{FJM2}.

An $\mc A_N$ module $W$ is 
said to be graded if 
$W=\oplus_{d_1,d_2\in\Z} W_{d_1,d_2}$,  $\dim W_{d_1,d_2}<\infty$, and for any $x\in \mc A_N$ of degree $a$ and weight $b$ we have  $x W_{d_1,d_2}\subset W_{d_1+a,d_2+b}$.

An $\mc A_N$ module $W$ is 
said to be admissible if it is graded and  for each $d_2\in \Z$ there exists $d$ (depending on $d_2$) such that such that $W_{d_1,d_2}=0$ for all $d_1>d$.

The $\E_1$ Fock representation $\F_1(v)$ of color $1$ can be extended to an admissible $\mc A_1$ module. The same is true for the $\check{\E}_1$ Fock representation $\F_1(v)$ of color $1$. This is the subject of the following lemma. 
\begin{lem}\label{F1 extended lemma}
The $\E_1$ Fock representation $\F_1(v)$ of color $1$ has an admissible $\mc A_1$ module structure such that $X^\pm_i(z)$, $\check{e}(z)$, $\check{f}(z)$ act by zero and $\check{\psi_0},\check{C}, K, \check{K}$ by $1$.

Similarly, the $\check{\E}_1$ Fock representation $\check{\F}_1(v)$ of color $1$ has an admissible $\mc A_1$ module structure such that $X^\pm_i(z)$, $e(z)$, $f(z)$ act by zero and $\psi_0,C, K, \check{K}$ by $1$. 
\end{lem}
\begin{proof}
The $[X^+_i,X^-_j]$ relation \eqref{EFshifted} follows from Lemma \ref{level 1 ideal}. Indeed, by  Lemma \ref{level 1 ideal}, on $\F_1(v)$ the operators $k^\pm_r(z)$ with $r\ge2$ act as $0$ . 
Therefore, the only non-zero contributions
in the right hand side of \eqref{EFshifted} are terms with 
$r\in\{0,1\}$, $i+j=\pm 1/2$.  In addition, by Lemma \ref{k on fock lem}, on $\F_1(v)$ we have $k_0^\pm(s_1z)+k_{1}^{\mp}(z)=0$.
Relation \eqref{EFshifted} follows.

The other relations are trivial.
\end{proof}
We call $\mc A_1$ modules described in Lemma \ref{F1 extended lemma} the Fock modules of color 1 and  denote them in the same way by $\mc F_1(v)$ and  $\check{\mc F}_1(v)$.

Note that $\check{\mc F}_1(v)$ can be obtained from $\mc F_1(v)$ via a twist by the isomorphism described in Lemma \ref{switch iso lem}.

By analogy with shifted quantum affine algebras, we may expect the existence of a  comultiplication of the form
$$\Delta_{N_1,N_2}:\ \mc A_{N_1+N_2}\to\mc A_{N_1}\otimes \mc A_{N_2},$$ 
such that $\Delta_{N_1,N_2}$ restricted to $\mc E_1$ and $\check{\mc E}_1$ is given by \eqref{E1 coproduct} (and the same formula for $\check{\mc E}_1$).

In this paper we do not give such a map in full generality and restrict ourselves to the case when $N_1=1$ (or $N_2=1$) and the algebra $\mc A_{N_1}=\mc A_1$ (or $\mc A_{N_2}=\mc A_1$) is reduced to its image in either $\End \mc F_1(v)$ or  $\End \check{\mc F}_1(v)$. In these cases the map $\Delta_{N_1,N_2}$ simplifies. We will discuss the issue of comultiplication further in \cite{FJM2}.

\begin{thm}\label{coproduct thm}
Let $\mathcal{M}$ be an admissible representation of $\mathcal{A}_{N}$.

Then the tensor product
${\mc F}_{1}(v)\otimes \mathcal{M}$ 
has a structure of an admissible representation of $\mathcal{A}_{N+1}$ such that the action of $\E_1,\check{\E}_1$ is given  by \eqref{E1 coproduct} and the action of $X^\pm_i(z)$ is given by
\begin{align*}
X^+_i(z)&\mapsto
k_0^-(z)\otimes X^+_{i-1/2}(s_1z)+k_1^-(q_1^{-1}z)\otimes X^+_{i+1/2}(s_1^{-1}z)
\,,\\
X^-_i(z)&\mapsto -1\otimes X^-_i(z)\,.
\end{align*}

Similarly, the tensor product
$\mathcal{M}\otimes\mc F_{1}(v)$ 
has a structure of an admissible representation of $\mathcal{A}_{N+1}^-$ such that the action of $\E_1,\check{\E}_1$ is given  by \eqref{E1 coproduct} and the action of $X^\pm_i(z)$ is given by
\begin{align*}
X^+_i(z)&\mapsto X^+_i(z)\otimes 1\,, \\
X^-_i(z)&
\mapsto X^-_{i+1/2}(s_1z)\otimes k_0^+(z)+X^-_{i-1/2}(s_1^{-1}z)\otimes k_1^+(q_1^{-1}z)
\,.
\end{align*}
\end{thm}
We prove Theorem \ref{coproduct thm} in Appendix, see Section \ref{coproduct proof sec}.

We recall that algebra $\mc A_{N+1}^-$ is isomorphic to $\mc A_{N+1}$, see Lemma \ref{shift lem}.  For generic $v$, we expect that the $\mc A_{N+1}$ modules 
${\mc F}_{1}(v)\otimes \mathcal{M}$ and
$\mathcal{M}\otimes\mc F_{1}(v)$ are isomorphic, cf. \cite{FJM2}. 

We have a similar statement for $\check{\F}_1(v)$ modules.

\begin{cor}\label{check cor}
Let $\mathcal{M}$ be an admissible representation of $\mathcal{A}_{N}$.

Then the tensor product
$\check{\mc F}_{1}(v)\otimes \mathcal{M}$ 
has a structure of an admissible representation of $\mathcal{A}_{N+1}$ such that the action of $\E_1,\check{\E}_1$ is given  by \eqref{E1 coproduct} and the action of $X^\pm_i(z)$ is given by
\begin{align*}
X^+_i(z)&\mapsto
\check{k}_0^-(q_1^iz)\otimes X^+_{i-1/2}(z)+\check{k}_1^-(q_1^{i+1}z)\otimes X^+_{i+1/2}(z)
\,,\\
X^-_i(z)&\mapsto -1\otimes X^-_i(z)\,.
\end{align*}

Similarly, the tensor product
$\mathcal{M}\otimes\check{\mc F}_{1}(v)$ 
has a structure of an admissible representation of $\mathcal{A}_{N+1}^-$ such that the action of $\E_1,\check{\E}_1$ is given  by \eqref{E1 coproduct} and the action of $X^\pm_i(z)$ is given by
\begin{align*}
X^+_i(z)&\mapsto X^+_i(z)\otimes 1\,, \\
X^-_i(z)&
\mapsto X^-_{i+1/2}(z)\otimes \check{k}_0^+(q_1^{-i}z)+X^-_{i-1/2}(z)\otimes \check{k}_1^+(q_1^{-i+1}z)
\,.
\end{align*}

\end{cor}
\begin{proof}
 The corollary is obtained from Theorem \ref{coproduct thm} by applying the isomorphism of Lemma \ref{switch iso lem}.
\end{proof}

The algebra $\mc A_0$ has the same admissible modules as $\E_2[K]$ by Theorems \ref{iota thm}, \ref{tau thm} and, therefore, has a rich representation theory, see for example \cite{FJMM2}. Using Theorem \ref{coproduct thm} or Corollary \ref{check cor}, one constructs a large family of representations of $\mc A_N$ for all dominant cases $N>0$. 

In particular, for even $N$, the representation $(\otimes_{i=1}^N \mc F_1(v_i))\otimes \mathcal V$ of algebra $\mc A_N$, where $\mc V$ is the Wakimoto module realized in four bosons, see \cite{S}, appeared in \cite{FJM1} as a rationalization of an extension of the deformed $W$-algebra of type $\gl(N+2|1)$.

\section{Appendix: Proof of Theorem \ref{iota thm}}\label{app iota sec}

\subsection{Relations}
The checks that $\iota$ respects quadratic relations is a brute force computation.

For example, $k_0^\pm (z)=K^\pm \exp(\sum_{i>0}a_{\pm i}^\pm z^{\mp i})$, 
where 
$$[a_r^+,a_{-r}^-
]=\frac{(1-q_2^r)(1-q_3^r)}
{(q^{-r}_1-1)}
\ \frac{C^r-C^{-r}}{r},$$
see \eqref{k0 expl}. In particular, $k_0^+(z)$ has elliptic commutation relations with $k_0^-(w)$. Then the modes of the current $\iota(\Psi_0(z))=k_0^\pm(z)\check{k}_0^\pm(z)$ satisfy
\begin{align*}
[a_r^++\check{a}_r^+,a_{-r}^-+\check{a}_{-r}^-]=\Big(\frac{(1-q_2^r)(1-q_3^r)}{q_1^{-r}-1}+\frac{(1-\check{q}_2^r)(1-\check{q}_3^r)}
{(\check{q}^{-1}_1-1)}\Big)\frac{C^r-C^{-r}}{r}=(q_2^r-q_2^{-r})\frac{C^r-C^{-r}}{r}.
\end{align*}
Therefore, $\iota$ preserves the relation between $\Psi_0^+(z)$ with $\Psi_0^-(w)$.

\newcommand{\Ad}{\mathop{\mathrm{Ad}}}

As another example, we have
\begin{align*}
\Ad\Bigl(\frac{k_0^-(q_1z)\check k^-_0(q_1^2z)}
{k_0^-(z)\check k^-_0(q_1z)}
\Bigr) X_0^+(w)
=\Ad\bigl(\psi^-_0(z){\check\psi_0^-(q_1^2z)}^{-1}\bigr)X^+_0(w)
=\frac{\omega^+(q_1z/w)}{\check{\omega}^+_1(q_1z/w)}
X^+_0(w)\,,
\end{align*}
where $\Ad(A)B=ABA^{-1}$.

Note that $\Ad(K\check{K})X_0^+(w)=q_2X_0^+(w)$. Solving the difference equation we obtain
\begin{align}\label{Xpsi}
\Ad\bigl((k_0^-(z)\check{k}^-_0(q_1z)
)^{-1}\bigr)X^+_0(w)
=q_2\frac{(1-q_1q_3z/w)(1-q_1^2q_3z/w)}{(1-z/w)(1-q_1z/w)}X^+_0(w)\,,
\end{align}
which means that $\iota$ preserves the relation between 
$\Psi_0^-(z)$ 
and $E_1(w)$.

We also have:
\begin{align*}
&(s_2-s_2^{-1})\iota\Big((z-\bar q_1w) (z-\bar q_3w)E_0(z)E_1(w)
- (w-\bar q_1z) (w-\bar q_3z)E_1(w)E_0(z)\Big)
\\
=&\,(z-\bar q_1w) (z-\bar q_3w)
\iota\bigl(\Psi_0^-(z)\bigr) X_{-1}^-(\bar q_1^{-1}Cz)X_0^+(w)
- (w-\bar q_1z) (w-\bar q_3z)X_0^+(w)\iota\bigl(\Psi_0^-(z)\bigr)X_{-1}^-(\bar q_1^{-1}Cz)\\
=&\,(z-\bar q_1w) (z-\bar q_3w)
\iota\bigl(\Psi_0^-(z)\bigr)[X_{-1}^-(\bar q_1^{-1}Cz),\,X_0^+(w)]\\
=&\,\iota\bigl(\Psi_0^-(z)\bigr)
(z-\bar q_1w) (z-\bar q_3w)
\Bigl\{\delta\Bigl(\frac{\bar q_1w}{z}\Bigr)
k_0^-(w)\check k^-_1(\check q_1^{-1}w)
+\delta\Bigl(\frac{\bar q_3w}{z}\Bigr)
k_1^-(q_1^{-1}w)\check k^-_1(w)
\Bigr\}\,=0.
\end{align*}
Here on the second step we used \eqref{Xpsi}.

Other quadratic relations are similar. We only note that
\begin{align*}    
&e(z)k_0^\pm(w)=\omega^\mp(C^{(1\pm1)/2}w/z) 
k_0^\pm(w)e(z),
\\
&k_0^{\pm}(w)f(z)=\omega^\mp\bigl(C^{(1\mp1)/2}w/z\bigr)f(z)k^\pm_0(w).
\end{align*}

\subsection{Serre relations}
The last two relations of $\E_2$ are quartic and called Serre relations. It is known that in admissible modules they are equivalent to cubic relations, see Lemma 2.1 in \cite{FJMM3}:
\begin{align}\label{serre}
    \mathop{\on{Sym}}_{z_1,z_2}\Big\{
    \bar q_1(z_1-\bar q_3 w) (z_2-\bar q_3 w) 
E_i(z_1)E_i(z_2)E_j(w)-
(1+\bar q_1\bar q_3)
(z_1-\bar q_3 w)(\bar q_1 z_2-w)E_i(z_1)E_j(w)E_1(z_2)\nn\\+\bar q_3(\bar q_1z_1-w)(\bar q_1z_2-w)E_j(w)E_i(z_1)E_i(z_2)
    \Big\}=0
\end{align}
and 
\begin{align*}
    \mathop{\on{Sym}}_{z_1,z_2}\Big\{
    \bar q_1(z_1-\bar q_3 w)(z_2-\bar q_3 w) 
F_j(w)F_i(z_1)F_i(z_2)-
(1+\bar q_1\bar q_3)
(z_1-\bar q_3 w)(\bar q_1 z_2-w)F_i(z_1)F_j(w)F_1(z_2)\\+\bar q_3(\bar q_1z_1-w)(\bar q_1z_2-w)F_i(z_1)F_i(z_2)F_j(w)
    \Big\}=0,
\end{align*}
where $i\neq j$, and the relations where the roles of 
$\bar q_1$ and $\bar q_3$ are interchanged.

We apply $\iota$ and check that the result is zero. For example, apply $\iota$ to the left hand side of \eqref{serre} with $i=1$ and $j=0$ and move $k_0^-$, $\check{k}_0^-$ to the left. The result is
\begin{align}\label{aux 2 eq}    
\frac{\big(\psi_0 k_0^-(\bar q_1^{-1}w) \check{k}_0^-(\bar q_3^{-1}w)\big)^{-1} }{s_2-s_2^{-1} }
&\frac{(\bar q_1 z_1-w)(\bar q_1z_2-w)}{(z_1-\bar q_1w)(z_2-\bar q_1w)}\Big(A(z_1,z_2,w)+A(z_2,z_1,w)\Big),
    \end{align}
where   
\begin{align*}
    A(z_1,z_2,w)=&f_1(z_1,z_2,w) X^+_0(z_1)X^+_0(z_2)X^-_{-1}(\bar q_1^{-1}Cw)\\&+f_2(z_1,z_2,w) X^+_0(z_1)X^-_{-1}(\bar q_1^{-1}Cw)X^+_0(z_2)+f_3(z_1,z_2,w) X^-_{-1}(\bar q_1^{-1}Cw)X^+_0(z_1)X^+_0(z_2),
    \end{align*}

\begin{align*}    
   f_1(z_1,z_2,w)= &\bar q_1(\bar q_3z_1-w)(\bar q_3z_2-w),\\f_2(z_1,z_2,w)=&-(1+\bar q_1\bar q_3)(\bar q_3 z_1-w)(z_2-\bar q_1 w),\\ f_3(z_1,z_2,w)=&\bar q_3(z_1-\bar q_1 w)(z_2-\bar q_1w).
\end{align*}

Let 
\begin{align*}
A'(z_1,z_2,w)=
\sum_{i=1}^3f_i(z_1,z_2,w)\cdot X_0^+(z_1)X^+_0(z_2)X^-_{-1}(\bar q_1^{-1}Cw)\,.
\end{align*}
The identity
\begin{align*}
\sum_{i=1}^3f_i(z_1,z_2,w) =-q_2^{-1}(1-q_2^{-1})(z_1-q_2z_2)w
\end{align*}
shows that $A'(z_1,z_2,w)+A'(z_2,z_1,w)=0$.
 
We write 
\begin{align*}
A(z_1,z_2,w)+A(z_2,z_1,w)=A(z_1,z_2,w)-A'(z_1,z_2,w)+ A(z_2,z_1,w)-A'(z_2,z_1,w)
\end{align*}
and simplify using the relation
\begin{align*}
[X^-_{-1}(\bar q_1^{-1}Cw), X^+_0(z)] 
=k_0^-(\bar q_1^{-1}w) \check{k}_0^-(\bar q_3^{-1}w)
\Bigl\{\delta\Bigl(\frac{\bar q_1z}{w}\Bigr)\check c_1\check e(\bar q_3^{-1}w)
+
\delta\Bigl(\frac{\bar q_3z}{w}\Bigr) c_1 e(\bar q_1^{-1}w)
\Bigr\}\,.
\end{align*}
Then up to an overall constant and $\psi_0$, 
\eqref{aux 2 eq} becomes 
\begin{align*}
\delta\Bigl(\frac{\bar q_3z_2}{w}\Bigr)& (\bar q_1z_2-w) 
\times\Bigl\{
-(1+\bar q_1\bar q_3)(z_1-\bar q_3w)
X^+_0(z_1) e(\bar q_1^{-1}w)\\
&+\bar q_3 (\bar q_1z_1-w)
e(\bar q_1^{-1}w)X^+_0(z_1)
+\bar q_3 \frac{z_1-\bar q_3 w}{\bar q_3z_1-w}(z_1-\bar q_1w)
X^+_0(z_1)e(\bar q_1^{-1}w)
\Bigr\}+ (z_1\leftrightarrow z_2).
\end{align*}
In the second term we move $X_0^+(z_i)$ to the left. Note that the delta term in $(eX^+_i)$ disappears due to the presence of prefactor $(\bar q_1z_i-w)$. 
Using the idenitity of raional functions
$$
-(1+\bar q_1\bar q_3)(z-\bar q_3w)+\bar q_3 (\bar q_1z-w)\omega(\bar q_1z/w)+\bar q_3 \frac{z-\bar q_3 w}{\bar q_3z_1-w}(z-\bar q_1w)= 0,
$$
with $z=z_i$ we see that the coefficient of the delta function $\delta(\bar q_3z_i/w)$ vanishes and therefore \eqref{aux 2 eq} is zero.

\subsection{Surjectivity and restriction to $\E_1\otimes \check{\E}_1$.}
Algebra $\mc A_0$ is generated by $K,\check{K}$, $\psi_0$, $C$, $k_0^\pm(z)$, $\check{k}_0^\pm(z)$, $e(z)$, $f(z)$, $\check{e}(z)$, $\check{f}(z)$, and $X_0^\pm(z)$. All of the above except maybe $e(z)$, $f(z)$, $\check{e}(z)$, $\check{f}(z)$  are in the image of $\iota$ for obvious reasons. 

Therefore to prove that $\iota$ is surjective, it is sufficient to show that $\iota$ maps $e(z)\mapsto e(z)$, $f(z)\mapsto f(z)$, $\check{e}(z)\mapsto \check{e}(z)$, $\check{f}(z)\mapsto \check{f}(z)$.

Let us show  $e(z)\mapsto e(z)$. Relation  \eqref{EF} gives
\begin{align*}
(1-\bar q_3 w/z)X_0^+(w)X_{-1}^-(\bar q_1^{-1}Cz)\Bigl|_{z=\bar q_3 w}
=-\psi_0k_1^-(q_1^{-1}w)\check k^-_0(w)\,.
\end{align*}
Hence 
\begin{align*}
\iota(E_1(w)E_0(\bar q_3w)) 
&=\frac{1}{s_2-s_2^{-1}}\iota\bigl(\Psi_0^-(z)\bigr)k_1^-(q_1^{-1}w)\check k^-_0(w)
q_2^{-1}\frac{(1-\bar q_1^{-1}\bar q_3)}{(1-\bar q_1\bar q_3)(1-\bar q_3^2)}
=a_{12}^{-1}e(q_1^{-1}w)\,,
\end{align*}
 which means $\iota\big(e(z)\big)=e(z)$. 
The other three checks are similar.

\section{Appendix: Proof of Theorem \ref{tau thm}. }\label{app tau sec}
In this section we prove Theorem \ref{tau thm}. For that we compute the relations between currents $X_i^\pm$, algebras $\E_1,\check{\E}_1$. All computations in the proof are inside of $\E_2$ acting on admissible module $\mc{M}_2$.

After we know $\tau$ is well-defined, the surjectivity is clear due to Lemma \ref{X simpl}.
\subsection{The first group of relations.}
We start by checking the first group of the relations, that is the adjoint action of $\E_1,\check{\E}_1$ on $X^\pm_i$.

The relations $(\psi^\pm X^\pm_i), (\check\psi^\pm X^\pm_i)$,  $(\psi^\pm X^\mp_i), (\check\psi^\pm X^\mp_i)$ are straightforward.

\begin{lem}\label{lem:seed}
The relations 
$(eX^+_1)$, $(\check eX^+_1)$, $(eX^-_0)$, $(\check eX^-_0)$, 
$(fX^+_0)$, $(\check fX^+_0)$, 
$(fX^-_{-1})$, $(\check fX^-_{-1})$, 
are satisfied.

\end{lem}
\begin{proof}
Consider
($eX^+_1$). 
The left hand side can be written as
\begin{align*}
-a E_1(\bq_1\bq_3^{-1}u)
&[E_0(\bq_1u),F_0(C^{-1}\bq_1z)]\Psi_0^+(C^{-1}\bq_1z)^{-1} 
\\
&=\frac{1}{s_1-s_1^{-1}} E_1(\bq_1\bq_3^{-1}u)
\Bigl\{\delta\Bigl(\frac{z}{u}\Bigr)
-\delta\Bigl(\frac{C^2u}{z}\Bigr)\Psi_0^-(\bq_1u)
\Psi_0^+(C\bq_1u)^{-1} 
\Bigr\}\,.
\end{align*}
Taking into account that $E_1(u)\Psi_0^-(\bq_3u)=0$, we obtain $(eX^+_1)$. 

Consider $(fX^+_0)$. The left hand side can be written as 
we compute
\begin{align*}
-aF_0(\bq_1 u)[F_1(\bq_1\bq_3^{-1}u), E_1(z)]
&=-\frac{1}{s_1-s_1^{-1}}F_0(\bq_1u) 
\Bigl\{\delta\Bigl(\frac{Cq_1u}{z}\Bigr)\Psi^+_1(q_1u)
-\delta\Bigl(\frac{Cz}{q_1u}\Bigr)\Psi^-_1(z)
\Bigr\}\,.
\end{align*}
We use $F_0(\bq_3u)\Psi^-_1(C^{-1}u)=0$ and 
$X_1^+(Cu)\psi^+(u)=F_0(\bq_1u)\Psi^+_1(q_1u)$.

Proof of other relations are similar. 
\end{proof}

Then the general case follows by induction using the recursion of Corollary \ref{inductive X cor}. 

For example, assume $i\leq 0$ and $(eX_i^+)$ and $(\check{e}X_i^+)$ hold. Let us show $(eX_{i-1}^+)$ and $(\check{e}X_{i-1}^+)$.
By Corollary \ref{inductive X cor}, we have
$X^+_{i-1}(z)=c_1X^+_i(q_1^{-1}z)e(q_1^{-1}z)=\check c_1 X^+_i(z)\check{e}(q_1^iz).$

Since $e(u)$ and $\check{e}(v)$ commute, 
we can use the second formula and $(eX_i^+)$ to deduce $(eX_{i-1}^+)$:
\begin{align*}
e(u)X^+_{i-1}(z)-\omega^-\Bigl(\frac{z}{u}\Bigr)X^+_{i-1}(z)e(u)
&=\Bigl(e(u)X^+_{i}(z)-\omega^-\Bigl(\frac{z}{u}\Bigr)X^+_{i}(z)e(u)\Bigr)
\check{c}_1\check{e}(q_1^iz)\\
&=\frac{1}{s_1^{-1}-s_1 
}
\delta\Bigl(\frac{z}{u}\Bigr)X^+_{i-1}(q_1z)
\check{c}_1
\check{e}(q_1^iz)=\frac{1}{s_1^{-1}-s_1 
}\delta\Bigl(\frac{z}{u}\Bigr)X_{i-2}^+(q_1z)\,.
\end{align*}
Similarly we obtain $(\check{e}X_{i-1}^+)$:
\begin{align*}
\check{e}(u)X^+_{i-1}(z)-
\check\omega^-\Bigl(q_1^{i-1}\frac{z}{u}\Bigr)X^+_{i-1}(z)\check e(u)
&=\frac{1}{\check s_1^{-1}-\check s_1
}\delta\Bigl(q_1^{i-1}\frac{z}{u}\Bigr)
X^+_{i-1}(q_1^{-1}z)c_1 e(q_1^{-1}z)\\ 
&=\frac{1}{\check s_1^{-1}-\check s_1 
}\delta\Bigl(q_1^{i-1}\frac{z}{u}\Bigr)
X^+_{i-2}(z)\,.
\end{align*}

\medskip 

As another example, assume $(eX^+_{i})$ with $i\ge0$ holds. Let us show 
$(eX^+_{i+1})$. By Corollary \ref{inductive X cor},
$X_{i+1}^+(z)=-\check{c_1}\check{f}(C^{-1}q_1^{i+1}z)\psi^+(C^{-1}q_1^{i+1}z)^{-1}X_i^+(z)$. 
We compute
\begin{align*}
&e(u)X^+_{i+1}(z)-\omega^-_1\Bigl(\frac{z}{u}\Bigr) X^+_{i+1}(z)e(u)\\
&=-\check c_1 \check f(C^{-1}q_1^{i+1}z){\check\psi^+
(C^{-1}q_1^{i+1}z)}^{-1}
\times 
\Bigl(e(u)X^+_{i}(z)-\omega^-_1\Bigl(\frac{z}{u}\Bigr) X^+_{i}(z)e(u)\Bigr)
\\
&=-\check c_1 \check f(C^{-1}q_1^{i+1}z){\check\psi^+
(C^{-1}q_1^{i+1}z)}^{-1}
\times \frac{1}{s_1^{-1}-s_1}\delta\Bigl(\frac{z}{u}\Bigr)X^+_{i-1}(q_1z)=\frac{1}{s_1^{-1}-s_1}\delta\Bigl(\frac{z}{u}\Bigr)X^+_{i}(q_1z)\,.
\end{align*}

The other cases are similar.

In addition, we have the following improvement on Corollary \ref{inductive X cor}. 
\begin{cor}\label{cor:rec-all}
The recursion relations in Corollary \ref{inductive X cor} hold for all $i\in\Z$.
\end{cor}
\begin{proof} The corollary is obtained by taking residues of the relations $(eX^\pm_i)$, $(\check eX^\pm_i)$, 
$(fX^\pm_i)$, $(\check fX^\pm_i)$.
\end{proof}
\subsection{The second group of relations.}
Next, we discuss the second group of relations, the relations between $X^\pm_i(z)$ and $X^\pm_j(w)$.

We consider the relations between $X^+_i(z)$  and $X^+_j(w)$. Namely, for $j\geq i$ we prove 
\begin{align}\label{i,j ++ rel}
\gamma_{i,j}(z,w)X^+_i(z)X^+_j(w)=(-1)^{j-i-1}
\gamma_{j,i}(w,z)X^+_j(w)X^+_i(z)
\end{align}
where $\gamma_{i,i}(z,w)=z-q_2w$ and 
\begin{align}\label{gamma}
\gamma_{i,j}(z,w)=\prod_{r=1} ^{j-i-1}(z-q_1^rw)\,,\qquad \gamma_{j,i}(w,z)=q_2p_{j-i}^+(z/w)
\prod_{r=0}^{j-i-2}(q_3^{-1}q_1^{r}w-z)\,,
\end{align}
for $j>i$.

Using the recursion in Corollary \ref{inductive X cor}, 
it is straightforward to check that 
 \eqref{i,j ++ rel} 
holds as rational functions.
For brevity we say that $a(z)b(w)$ is a Laurent polynomial if its matrix coefficients
are all Laurent polynomials. 
We are to show that the left hand side of  \eqref{i,j ++ rel} is a Laurent polynomial. 

The results of the previous section imply in particular that
\begin{align*}
&(u-z)e(u)X^+_i(z)\,,\quad (u-q_1^iz)\check e(u)X^+_i(z) \,,
\\
&(z-q_2^{-1}u)(z-q_3^{-1}u)X^+_i(z)e(u)\,,
\quad
(q_1^{i}z-\check q_2u)(q_1^{i}z-\check q_3 u)X^+_{i-1}(z)\check e(u)\,,
\end{align*}
are Laurent polynomials. 
Using this we show first that 
\begin{align}\label{ii, i-1i}
(z-q_2w)X^+_{i}(z)X^+_{i}(w),
\quad 
X^+_{i-1}(z)X^+_i(w)
\end{align}
are Laurent polynomials for all $i\in\Z$. 

Suppose that $(z-q_2w)X^+_i(z)X^+_i(w)$ is a Laurent polynomial 
for some $i$.
Using again the recursion in Corollary \ref{inductive X cor}, 
we see that
\begin{align*}
&(q_1^{-1}z-q_2w)(q_1^{-1}z-w)
X^+_{i-1}(z)X^+_i(w)
=c_1(q_1^{-1}z-q_2w)(q_1^{-1}z-w) X^+_i(q_1^{-1}z)e(q_1^{-1}z)X^+_i(w)\,,
\\
&(z-q_2w)(z-w)X^+_{i-1}(z)X^+_i(w)
=\check c_1(z-q_2w)(z-w)
X^+_i(z)\check e(q_1^{i}z)X^+_i(w)\,,
\end{align*}
are both Laurent polynomials. It follows that $X^+_{i-1}(z)X^+_i(w)$ is a Laurent polynomial. 
This implies further that 
\begin{align*}
&(z-q_2w)(z-q_3w)
X_{i-1}^+(z)X^+_{i-1}(w)= 
c_1(z-q_2w)(z-q_3w)
X^+_{i-1}(z)X^+_{i}(q_1^{-1}w)e(q_1^{-1}w)
\\
&(z-\check q_2w)(z-\check q_3w)X_{i-1}^+(z)X^+_{i-1}(w)
= 
(z-\check q_2w)(z-\check q_3w)
X^+_{i-1}(z)X^+_{i}(w)\check e(q_1^{i}w)
\end{align*}
are both Laurent polynomials. Hence $(z-q_2w)X_{i-1}^+(z)X^+_{i-1}(w)$ is a Laurent polynomial. 
Since  $(z-q_2w)X_{0}^+(z)X^+_{0}(w)$ is a Laurent polynomial, we have shown that
\eqref{ii, i-1i} are Laurent polynomias for $i\le0$.

Arguing similarly using $f(u),\check{f}(u)$, we see 
that if $(z-q_2w)X^+_i(z)X^+_i(w)$ is a Laurent polynomial then so are
$X^+_i(z)X^+_{i+1}(w)$, $(z-q_2w)X^+_{i+1}(z)X^+_{i+1}(w)$. 
This shows that \eqref{ii, i-1i} are Laurent polynomials
for $i\ge1$.

By the same argument, it is easy to verify that
if $\gamma_{i,j}(z,w)X^+_i(z)X^+_j(w)$ 
is a Laurent polynomial for some $i,j$ with $i<j$,
then so is $\gamma_{i-1,j}(z,w)X^+_{i-1}(z)X^+_j(w)$. 

This completes the proof of the relation \eqref{i,j ++ rel}.

The relations between $X^-_i(z)$  and $X^-_j(w)$
are checked in the same way.

\subsection{The third group of relations.}
We finish with checking relations \eqref{EF} between $X^+_i$ and $X^-_j$. We again use an inductive argument. 

\begin{lem}\label{EF recursion lem}
If \eqref{EF} holds 
for some $i,j\in\Z$, $i+j\neq 1$ then \eqref{EF} holds 
for $i-1,j$ and for $i,j-1$.

If \eqref{EF} holds 
for some $i,j\in\Z$, $i+j\neq -1$, then \eqref{EF} holds 
for $i+1,j$ and for $i,j+1$.
\end{lem}

\begin{proof}
 We compute $e(u)(\sharp)_{i,j}-\omega^-(z/u)(\sharp)_{i,j}e(u)$, where $(\sharp)_{i,j}$ is the relation $\eqref{EF}$ for $[X^+_i,X_j^{-}]$.
For the left hand side we write
\begin{align*}
e(u)[X_i^+(z),X^-_j(w)] 
-\omega^-\Bigl(\frac{z}{u}\Bigr)[X_i^+(z)&,X^-_j(w)] e(u)\\
=[\Bigl(e(u)X^+_i(z)&-\omega^-(z/u)X^+_i(z)e(u)\Bigr)
,X^-_j(w)]\\
&
+\omega^- 
(z/u)X_i^+(z)[e(u),X^-_j(w)]-[e(u),X^-_j(w)]X^+_i(z)
\end{align*}
and substitute
\begin{align*}
&e(u)X^+_i(z)-\omega^-(z/u)X^+_i(z)e(u)=
-\frac{c_1}{\kappa_1}
\delta\Bigl(\frac{z}{u}\Bigr)X^+_{i-1}(q_1z)
\,, \\
&[e(u),X_j^-(w)]=\frac{c_1}{\kappa_1}
\delta\Bigl(\frac{Cq_1u}{w}\Bigr)\psi^-(u) X^-_{j-1}(q_1^{-1}w) 
\,.
\end{align*}
For the right hand side we use Lemma \ref{k identity lemma}. We obtain 
\begin{align*}
&-\delta\Bigl(\frac{z}{u}\Bigr)[X^+_{i-1}(q_1z),X^-_j(w)]
+\delta\Bigl(\frac{Cq_1u}{w}\Bigr)\psi^-(u)[X^+_i(z),X^-_{j-1}(q_1^{-1}w)]
\\ 
&
=\psi_0^{-i}
\sum_{r+\check{r}=i+j\atop r,\check{r}\in\Z_{\ge0}}
\delta\Bigl(\frac{C w}{q_1^rz}\Bigr) 
\,
\Bigl(
\delta\Bigl(\frac{C q_1 u}{w}\Bigr)
\psi^-(u)k^+_{r-1}(q_1^{-r}w)
-\delta\Bigl(\frac{z}{u}\Bigr)
\psi_0
k^+_{r-1}(q_1^{-r+1}w)
\Bigr)
\check{k}_{\check{r}}^+(\check{q}_1^{-\check{r}}q_1^{-j}w)  \\
&-\psi_0^{-j}
\sum_{r+\check{r}=-i-j\atop r,\check{r}\in\Z_{\ge0}}
\delta\Bigl(\frac{C z}{q_1^rw}\Bigr)
\,
\Bigl(\delta\Bigl(\frac{C q_1u}{w}\Bigr)
\psi_0
\psi^-(u)k^-_{r+1}(q_1^{-r-1}z)
-\delta\Bigl(\frac{z}{u}\Big)k^-_{r+1}(q_1^{-r}z) 
\Bigr)
\check{k}_{\check{r}}^-(\check{q}_1^{-\check{r}}q_1^iz). 
\end{align*}
Therefore, 
\begin{align*}
[X^+_{i-1}(q_1z),X^-_j(w)]
=& \psi_0^{-i+1}
\sum_{r+\check{r}=i+j\atop r,\check{r}\in\Z_{\ge0}}
\delta\Bigl(\frac{C w}{q_1^rz}\Bigr) 
\, k_{r-1}^+(q_1^{-r+1}w)
\check{k}_{\check{r}}^+(\check{q}_1^{-\check{r}}q_1^{-j}w)      
\\
&-\psi_0^{-j}
\Bigl\{
\sum_{r+\check{r}=-i-j\atop r,\check{r}\in\Z_{\ge0}}
\delta\Bigl(\frac{C z}{q_1^rw}\Bigr)
\,k^-_{r+1}(q_1^{-r}z) 
\check{k}_{\check{r}}^-(\check{q}_1^{-\check{r}}q_1^iz) 
+\delta\Bigl(\frac{C q_1 z}{w}\Bigr) 
\psi_0
\psi^-(z)Y(z)
\Bigr\}\,,
\end{align*}
\begin{align*}
[X^+_i(z),X^-_{j-1}(q_1^{-1}w)]
=&\psi_0^{-i}
\sum_{r+\check{r}=i+j\atop r,\check{r}\in\Z_{\ge0}}
\delta\Bigl(\frac{C w}{q_1^rz}\Bigr) 
\, k_{r-1}^+(q_1^{-r}w)
\check{k}_{\check{r}}^+(\check{q}_1^{-\check{r}}q_1^{-j}w)    
\\
&-\psi_0^{-j+1}
\Bigl\{
\sum_{r+\check{r}=-i-j\atop r,\check{r}\in\Z_{\ge0}}
\delta\Bigl(\frac{C z}{q_1^rw}\Bigr)
\,k^-_{r+1}(q_1^{-r-1}z) 
\check{k}_{\check{r}}^-(\check{q}_1^{-\check{r}}q_1^iz)  
+\delta\Bigl(\frac{C q_1 z}{w}\Bigr) Y(z)
\Bigr\}
\,,\nn
\end{align*}
for some $Y(z)$.

Similarly, computing $\check{e}(u)(\sharp)_{i,j}-\check{\omega}^-(q_1^iz/u)(\sharp)_{i,j}\check{e}(u)$, we arrive at
\begin{align*}
[X^+_{i-1}(z),X^-_j(w)]
=&\psi_0^{-i+1}
\sum_{r+\check{r}=i+j\atop r,\check{r}\in\Z_{\ge0}}
\delta\Bigl(\frac{C w}{q_1^rz}\Bigr) 
\, k_{r}^+(q_1^{-r}w)
\check{k}_{\check{r}-1}^+(\check{q}_1^{-\check{r}+1}q_1^{-j}w)      
\\
&-\psi_0^{-j}
\Bigl\{
\sum_{r+\check{r}=-i-j\atop r,\check{r}\in\Z_{\ge0}}
\delta\Bigl(\frac{C z}{q_1^rw}\Bigr)
\,k^-_{r}(q_1^{-r}z) 
\check{k}_{\check{r}+1}^-(\check{q}_1^{-\check{r}}q_1^iz) 
+\delta\Bigl(\frac{C z}{q_1^{-i-j+1}w}
\Bigr) 
\psi_0
\check{\psi}^-(q_1^iz)
\check{Y}(z)
\Bigr\} 
\,,
\\
[X^+_i(z),X^-_{j-1}(w)]
=&\psi_0^{-i}
\sum_{r+\check{r}=i+j
\atop r,\check{r}\in\Z_{\ge0}}
\delta\Bigl(\frac{C w}{q_1^rz}\Bigr) 
\, k_{r}^+(q_1^{-r}w)
\check{k}_{\check{r}-1}^+(\check{q}_1^{-\check{r}}q_1^{-j}w)  
\\
&-\psi_0^{-j+1}
\Bigl\{
\sum_{r+\check{r}=-i-j\atop r,\check{r}\in\Z_{\ge0}}
\delta\Bigl(\frac{C z}{q_1^rw}\Bigr)
\,k^-_{r}(q_1^{-r}z) 
\check{k}_{\check{r}+1}^-(\check{q}_1^{-\check{r}-1}q_1^iz)  
+\delta\Bigl(\frac{C z}{q_1^{-i-j+1}w}\Bigr) 
\check{Y}(z)
\Bigr\}
\,,\nn
\end{align*}
for some  $\check{Y}(z)$.
Comparing these and using the assumption $i+j\neq 1$, 
we conclude 
\begin{align*}
&Y(z)=k^-_0(z)\check{k}^-_{-i-j+1}(\check{q}^{i+j-1}q_1^iz) \,,\qquad
\check{Y}(z)=k^-_{-i-j+1}(q_1^{i+j-1}z)\check{k}_0(q_1^iz)\,,
\end{align*}
which proves the first assertion of the lemma.

The second assertion is proved similarly by computing and comparing
$\omega^+(w/u)f(u)(\sharp)_{i,j}-(\sharp)_{i,j}f(u)$ 
and 
$\check{\omega}^+(q_1^{-i}w/u)\check{f}(u)(\sharp)_{i,j}-(\sharp)_{i,j}f(u)$.
\end{proof}

Now we checked that \eqref{EF} holds for $i=j=0$. Using Lemma \ref{EF recursion lem}  we see that \eqref{EF} hold for all $i,j$ where $i+j\neq 0$.
It remains to show that 
\begin{align}
[X_i^+(z),X_{-i}^-(w)]=
\psi_0^{-i} \delta\Bigl(\frac{Cw}{z}\Bigr)k_0^+(w)\check{k}_0^+(q_1^{i}w)
-\psi_0^{i} \delta\Bigl(\frac{Cz}{w}\Bigr)k_0^-(z)\check{k}_0^-(q_1^{i}z)\,.
\label{i+j=0}
\end{align}
From the proof of Lemma \ref{EF recursion lem}, we have 
\begin{align*}
&[X^+_{i-1}(z),X^-_{-i+1}(w)]=
\psi_0^{-i+1} \delta\Bigl(\frac{Cw}{z}\Bigr)k_0^+(w)\check{k}_0^+(q_1^{i-1}w)
-\psi_0^{i-1} \delta\Bigl(\frac{Cz}{w}\Bigr)\psi_0\check{\psi}^-(q_1^iz)
\check{Y}_i(z)\,,
\\
&[X^+_{i}(z),X^-_{-i}(w)]=
\psi_0^{-i} \delta\Bigl(\frac{Cw}{z}\Bigr)k_0^+(w)\check{k}_0^+(q_1^{i}w)
-\psi_0^{i} \delta\Bigl(\frac{Cz}{w}\Bigr)\check{Y}_i(z)\,
\end{align*}
for some $\check{Y}_i(z)$. 
Suppose \eqref{i+j=0} holds for some $i$. Then 
$\check{Y}_i(z)=k_0^-(z)\check{k}_0^-(q_1^{i}z)$,
which yields
\begin{align*}
\psi_0\check{\psi}^-(q_1^iz)\check{Y}_i(z)
=k_0^-(z)\check{k}_0^-(q_1^{i-1}z)\,,
\end{align*}
showing \eqref{i+j=0} for $i-1$.
Similarly, \eqref{i+j=0} for  $i-1$ implies $i$. 
Hence \eqref{i+j=0} holds for all $i$.

That finishes the check of all relations and therefore Theorem \ref{tau thm} is proved.

\section{Appendix: Proof of Theorem \ref{coproduct thm}}\label{coproduct proof sec}
In this section we prove Theorem \ref{coproduct thm}. 

We denote the expressions given in the right hand side as $\Delta X^\pm_i(z)$. 
We check that currents $\Delta X^\pm_i(z)$, 
$\Delta_1 e(z)$, etc, satisfy the relations of $\mc A_{N+1}$.

\subsection{The first group of relations.}
We start by checking the first group of the relations, that is the adjoint action of $\E_1,\check{\E}_1$ on $X^\pm_i$.

For example let us check relation $(eX_i^+)$. We need to
show  that on ${\mc F}_{1}(v)\otimes \mathcal{M}$ we have
\begin{align}
\Delta_1e(z)\cdot
\bigl(k_0^-(w)&\otimes X^+_{i-1/2}(s_1w)+k^-_1(q_1^{-1}w)\otimes\ X_{i+1/2}^+(s_1^{-1}w)\bigr)
\label{Delta e KPhi}\\
=  &\omega^-(w/z)
\bigl(k_0^-(w)\otimes X_{i-1/2}^+(s_1w)+ k^-_1(q_1^{-1}w)\otimes X^+_{i+1/2}(s_1^{-1}w)\bigr)
\cdot\Delta_1e(z)
\nn\\
&+\frac{1}{s_1^{-1}-s_1}\delta(w/z)
\Bigl\{k_0^-(q_1z)\otimes X^+_{i-3/2}(s_1^3z)+k_0^-(z)\otimes X^+_{i-1/2}(s_1w)
\Bigr\}\,.\nn
\end{align}
We have relation $(eX^+_{i+1/2})$ on $\mathcal{M}$ 
and 
\begin{align*}
&e(z)k^-_0(w)=\omega^+(w/z)k^-_0(w)e(z)\,,\\
&e(z)k^-_1(q_1^{-1}w)=\omega^-(w/z)k_1^-(q_1^{-1}w)e(z)\,
\end{align*}
on $\F_1(v)$.
Using these together with 
\begin{align*}
\omega^+(w/z)-\omega^-_1(w/z)= 
\frac{c_1}{s_1^{-1}-s_1}\Bigl(\delta(w/z)-\delta(q_1^{-1}w/z)\Bigr)\,,
\end{align*}
we compute
\begin{align*}
&\Delta_1 e(z)\cdot
\bigl(k_0^-(w)\otimes X^+_{i-1/2}(s_1w)+k^-_1(q_1^{-1}w)\otimes X^+_{i+1/2}(s_1^{-1}w)\bigr)
\\
&=e(z)k_0^-(w)\otimes X^+_{i-1/2}(s_1w)+e(z)k_1^{-}(q_1^{-1}w)\otimes 
X^+_{i+1/2}(s_1^{-1}w)\\
&+\psi^-(z)k^-_0(w)\otimes e(s_1z)X^+_{i-1/2}(s_1w)
+\psi^-(z)k_1^-(q_1^{-1}w)\otimes e(s_1z)X^+_{i+1/2}(s_1^{-1}w)
\\
&=\Bigl(\omega^-(w/z)+\frac{c_1}{s_1^{-1}-s_1}\bigl(\delta(w/z)
-\underline{\delta(q_1^{-1}w/z)}\bigr)\Bigr)
k_0^-(w)e(z)\otimes X^+_{i-1/2}(s_1w)
\\
&+\omega^-(w/z)k_1^{-}(q_1^{-1}w)\otimes X^+_{i+1/2}(s_1^{-1}w)
\\
&+k_0^-(w)\psi^-(z)\otimes
\Bigl\{\omega^-(w/z)X^+_{i-1/2}(s_1w)e(s_1z)+\frac{1}{s_1^{-1}-s_1}
\delta(w/z)X^+_{i+1/2}(s_1^3z)
\Bigr\}\\
&+\psi^-(z)k_1^-(q_1^{-1}w)\otimes
\Bigl\{\omega^-(q_1^{-1}w/z)X^+_{i-1/2}(s_1^{-1}w)e(s_1z)
+\frac{1}{s_1^{-1}-s_1}\underline{\delta(q_1^{-1}w/z)}X^+_{i-1/2}(s_1w)
\Bigr\}\,.
\end{align*}
The underlined terms cancel each other.
Finally using the relation
\begin{align*}
&\omega^-(q_1^{-1}w/z)\psi^-(z)k_1^-(q_1^{-1}w)
=\omega^-(w/z)k_1^-(q_1^{-1}w)\psi^-(z)
\end{align*}
 valid on $\F_1(v)$, we arrive at \eqref{Delta e KPhi}.

The other relations between $\mc K_N$ and $X^\pm$ are checked similarly.

\subsection{The second group of relations.}
Our goal is to prove the formal series relations
\begin{align*}
\gamma_{i,j}(z,w)\Delta X^+_{i-1/2}(s_1z)\Delta X^+_{j-1/2}(s_1w) 
=(-1)^{i-j-1}\gamma_{j,i}(w,z)\Delta X^+_{j-1/2}(s_1w)\Delta X^+_{i-1/2}(s_1z)\,,
\end{align*}
where $\gamma_{ij}(z,w)$ are given by \eqref{gamma}.
Here, for convenience we have shifted the indices in $\Delta X^\pm_i(z)$ to $\Delta X^\pm_{i-1/2}(z)$.
Without loss of generality we assume that $i\le j$. 

It is straightforward to check that these relations hold as rational functions. We need to show that the left hand side has no poles. 

We write
$\Delta X^+_{i-1/2}(s_1z)\Delta X^+_{j-1/2}(s_1w)=I+II+III+IV$,
where
\begin{align*}
&I=k^-_0(s_1z)k^-_0(s_1w)\otimes X^+_{i-1}(q_1z) X^+_{j-1}(q_1w)\,,
\\ 
&II=k^-_0(s_1z)k^-_1(s_1^{-1}w)\otimes X^+_{i-1}(q_1z) X^+_{j}(w)\,,
\\
&III=k^-_1(s_1^{-1}z)k^-_0(s_1w)\otimes X^+_{i}(z) X^+_{j-1}(q_1w)\,,
\\
&IV=k^-_1(s_1^{-1}z)k^-_1(s_1^{-1}w)
\otimes X^+_{i}(z) X^+_{j}(w)\,.
\end{align*}

Among them, two terms,
\begin{align*}
&\gamma_{i,j}(z,w)I=
k^-_0(s_1z)k^-_0(s_1w)\otimes \gamma_{i,j}(z,w)
X^+_{i-1}(q_1z) X^+_{j-1}(q_1w)\,,
\\
&\gamma_{i,j}(z,w)IV =
k^+_0(z)k^+_0(w)
\otimes \gamma_{i,j}(z,w)
X^+_{i}(z) X^+_{j}(w)\,,
\end{align*}
are clearly Laurent polynomials. 
Here in the second equality we used the identity $k_1^-(z)=-k_0^+(s_1z)$ which holds in Fock module $\mc F_1(v)$, see Lemma \ref{k on fock lem}.

We also have
\begin{align*}
k_1^-(s_1^{-1}z)k_0^-(s_1w)=k^-_0(s_1w)k_1^-(s_1^{-1}z)
\frac{1-q_2^{-1}w/z}{1-w/z} \frac{1-q_3^{-1}w/z}{1-q_1w/z}\,.
\end{align*}
Therefore,
\begin{align*}
&\gamma_{i,j}(z,w)(II+III)
=k^-_0(s_1w)k_1^-(s_1^{-1}z)\otimes (\ast)\,,
\\
&(\ast)=\gamma_{i,j}(z,w)
\Bigl\{X^+_{i-1}(q_1z) X^+_{j}(w)
+\omega\Bigl(\frac{q_1w}{z}\Bigr)
X^+_{i}(z) X^+_{j-1}(q_1w)
\Bigr\}\,.
\end{align*}
It suffices to show that ($\ast$) has no poles.

If $i=j$, we have
\begin{align*}
&\gamma_{i,i}(z,w)=(z-q_2w)\,,\quad
\gamma_{i-1,i}(z,w)=1\,,\quad
\gamma_{i,i-1}(z,w)=q_2^{-1}
\frac{(z-q_2w)(z-q_2q_1^{-1}w)}
{(z-w)(z-q_1^{-1}w)}\,.
\end{align*}
Hence $\gamma_{i,i}(z,w)X^+_{i-1}(q_1z) X^+_{i}(w)$ and
\begin{align*}
\gamma_{i,i}(z,w)\omega\Bigl(\frac{q_1w}{z}\Bigr)
X^+_{i}(z) X^+_{i-1}(q_1w)
=q_2
(z-q_2^{-1}w)\gamma_{i,i-1}(z,q_1w)
X^+_{i}(z) X^+_{i-1}(q_1w)
\end{align*}
are both Laurent polynomials.

For $i<j$, we first note that ($\ast$) can have a pole at most at $z=w$ and that pole is at most simple.
Indeed, for $i<j-1$ we have 
\begin{align*}
&\gamma_{i,j}(z,w)=q_1^{-j+i}\frac{1}{z-w}\gamma_{i-1,j}(q_1z,w)\,,\\ 
&\gamma_{i,j}(z,w)\omega\Bigl(\frac{q_1w}{z}\Bigr)
=\frac{(z-q_2^{-1}w)(z-q_3^{-1}w)}{z-w}\gamma_{i,j-1}(z,q_1w)\,,
\end{align*}
and  $\gamma_{i-1,j}(q_1z,w)X_{i-1}^+(q_1z)X_j^+(w)$ and  $\gamma_{i,j-1}(z,q_1w)X_{i-1}^+(z)X_j^+(q_1w)$ are Laurent polynomials.
For $i=j-1$, ($\ast$) becomes
\begin{align*}
X^+_{i-1}(q_1z)X^+_{i+1}(w) 
+\omega\Bigl(\frac{q_1w}{z}\Bigr)X^+_i(z)X^+_i(q_1w)\,.
\end{align*}
Since $X^+_i(q_1w)^2=0$, both terms can have at most a simple pole at $z=w$.

Note also that in all cases $\gamma_{i,j}(z,w)$ is regular at $z=w$. 

Therefore the check reduces to the following lemma.
\begin{lem}\label{lem:XXoXX}
For all $j>i$, we have
\begin{align*}
\mathop{\res}_{z=w}\Bigl\{X^+_{i-1}(z)X^+_j(q_1^{-1}w)
+\omega\Bigl(\frac{q_1w}{z}\Bigr)
X^+_i(q_1^{-1}z)X^+_{j-1}(w)
\Bigr\}\frac{dz}{z} =0\,.
\end{align*} 
\end{lem}
\begin{proof}
We use the recursion of Corollary \ref{inductive X cor} we have
\begin{align*}
X^+_{i-1}(z)X^+_j(q_1^{-1}w)
&=c_1 X_i^+(q_1^{-1}z)
e(q_1^{-1}z)X^+_j(q_1^{-1}w)\,.
\end{align*}
Using that $X^+_i(z)X^+_{j}(w)$ does not have a pole at $z=w$, we have
\begin{align*}
\mathop{\res}_{z=w}&
X^+_{i-1}(z)X^+_{j}(q_1^{-1}w)\frac{dz}{z}
=\mathop{\res}_{z=w}\,
c_1X_i^+(q_1^{-1}z)
e(q_1^{-1}z)X^+_{j}(q_1^{-1}w) \, \frac{dz}{z}\\
&=c_1 X_i^+(q_1^{-1}w)\cdot
\mathop{\res}_{z=w} \,
e(q_1^{-1}z)X^+_{j}(q_1^{-1}w) \, \frac{dz}{z}
=c_1X_i^+(q_1^{-1}w)
\Bigl(-\frac{1}{s_1-s_1^{-1}}\Bigr)X^+_{j-1}(w)
\\
&=\frac{(1-q_2)(1-q_3)}{1-q_1^{-1}}
X_i^+(q_1^{-1}w)X^+_{j-1}(w)=-\mathop{\res}_{z=w}\,\omega\Bigl(\frac{q_1w}{z}
\Bigr)\frac{dz}{z}
\cdot
X_i^+(q_1^{-1}w)X^+_{j-1}(w)\,.
\end{align*}
\end{proof}

\begin{rem}{\rm
For $U_q(\widehat{\mathfrak{gl}}_2)$,
the usual quadratic relation $(z-q^2w)E(z)E(w)+(w-q^2z)E(w)E(z)=0$
reads in components as
\begin{align*}
&[E_{i},E_{j-1}]_{q^2}+[E_{j},E_{i-1}]_{q^2}=0\,.
\end{align*}
The relation in Lemma \ref{lem:XXoXX}
can be viewed as an affine analog of these relations.
\qed}
\end{rem}

The $X^-_i(z)X^-_j(w)$ relations are similar.

\subsection{The third group of relations.}
The check of  relations \eqref{EFshifted} for $[X^+_i(z),X^-_j(w)]$ is straightforward using the action of $k_r^\pm(z)$ given in Lemma \ref{k on fock lem}
and the coproduct of $k^\pm_r(z)$ given in Lemma \ref{lem:DeltaK}, where $\psi_0$ is set to $1$.

\medskip

{\bf Acknowledgments.\ }
BF is partially supported by ISF 3037 2025 1 1. MJ is partially supported by JSPS KAKENHI Grant Numbers 25K07041, 23K03137,
25K06942.
EM is partially supported by Simons Foundation grant number \#709444.

BF and EM thank Rikkyo University, where a part of this work was done, for hospitality.

\end{document}